%% file: main.tex
\documentclass[11pt,reqno]{amsart}
\usepackage[utf8]{inputenc} 
\usepackage[T1]{fontenc}    
\usepackage{lmodern}  

\usepackage{graphics,color}
\usepackage{amssymb, amsmath, amsthm, amscd}
\usepackage{latexsym, verbatim, graphicx, amsfonts}
\usepackage[mathscr]{euscript}
\usepackage{amsmath, amsthm, amssymb}
\usepackage{dsfont}
\usepackage{enumitem}
\usepackage{mathrsfs}
\usepackage{mathtools}
\usepackage{float}

\usepackage{xcolor}                 
\usepackage[colorlinks=true, linkcolor=blue, citecolor=black, urlcolor=black]{hyperref}           

\usepackage{aliascnt}
\usepackage[nameinlink,noabbrev]{cleveref}

\theoremstyle{plain}
\newtheorem{theorem}{Theorem}[section]

\newaliascnt{lemma}{theorem}
\newtheorem{lemma}[lemma]{Lemma}
\aliascntresetthe{lemma}
\crefname{lemma}{lemma}{lemmas}
\Crefname{lemma}{Lemma}{Lemmas}

\newaliascnt{proposition}{theorem}
\newtheorem{proposition}[proposition]{Proposition}
\aliascntresetthe{proposition}
\crefname{proposition}{proposition}{propositions}
\Crefname{proposition}{Proposition}{Propositions}

\newaliascnt{corollary}{theorem}
\newtheorem{corollary}[corollary]{Corollary}
\aliascntresetthe{corollary}
\crefname{corollary}{corollary}{corollaries}
\Crefname{corollary}{Corollary}{Corollaries}

\theoremstyle{definition}
\newaliascnt{definition}{theorem}
\newtheorem{definition}[definition]{Definition}
\aliascntresetthe{definition}
\crefname{definition}{definition}{definitions}
\Crefname{definition}{Definition}{Definitions}

\newaliascnt{remark}{theorem}
\newtheorem{remark}[remark]{Remark}
\aliascntresetthe{remark}
\crefname{remark}{remark}{remarks}
\Crefname{remark}{Remark}{Remarks}

\newaliascnt{question}{theorem}

\aliascntresetthe{question}
\crefname{question}{question}{questions}
\Crefname{question}{Question}{Questions}

\newaliascnt{example}{theorem}
\newtheorem{example}[example]{Example}
\aliascntresetthe{example}
\crefname{example}{example}{examples}
\Crefname{example}{Example}{Examples}

\newcommand{\N}{\mathbb{N}}

\newcommand{\norm}[1]{ \| #1 \| }

\DeclareMathOperator{\supp}{supp}

\DeclareMathOperator{\diag}{diag}

\newcommand{\vertiii}[1]{{\left\vert\kern-0.25ex\left\vert\kern-0.25ex\left\vert #1 
    \right\vert\kern-0.25ex\right\vert\kern-0.25ex\right\vert}}

\font\sstext=ecss1000
\font\sssub=ecss1000 at 7pt
\font\sssubsub=ecss1000 at 5pt

\newfam\ssfam
\textfont\ssfam=\sstext
\scriptfont\ssfam=\sssub
\scriptscriptfont\ssfam=\sssubsub

\makeatletter
\def\@tocline#1#2#3#4#5#6#7{\relax
  \ifnum #1>\c@tocdepth
  \else
    \par \addpenalty\@secpenalty\addvspace{#2}%
    \begingroup
      \hyphenpenalty\@M
      \@ifempty{#4}{%
        \@tempdima\csname r@tocindent\number#1\endcsname\relax
      }{%
        \@tempdima#4\relax
      }%
      \parindent\z@
      \leftskip#3\relax
      \advance\leftskip\@tempdima\relax
      \rightskip\@pnumwidth plus4em
      \parfillskip-\@pnumwidth
      #5\leavevmode\hskip-\@tempdima
        \ifcase #1
        \or\or \hskip 1em \or \hskip 2em \else \hskip 3em \fi
        #6\nobreak\relax
      \dotfill\hbox to\@pnumwidth{\@tocpagenum{#7}}\par
      \nobreak
    \endgroup
  \fi
}
\makeatother



\begin{document}

\title[Primariness of the spaces $\ell_p(C(K))$ for $1 \leq p \leq \infty$]{Primariness of the spaces $\ell_p(C(K))$ for $1 \leq p \leq \infty$}
\author[A. Acuaviva]{Antonio Acuaviva}
\address{School of Mathematical Sciences,
Fylde College,
Lancaster University,
LA1 4YF,
United Kingdom} \email{ahacua@gmail.com}

\date{\today}

\begin{abstract}
    We prove that the spaces $\ell_p(C(\alpha))$ and $\ell_p(C[0,1])$ have the uniform primary factorisation property whenever $\alpha$ is an ordinal and $1<p\leq\infty$. For the case $p=1$, we establish a general criterion ensuring that $\ell_1(X)$ inherits the uniform primary factorisation property from $X$. As a consequence, $\ell_p(C(K))$ is primary for every compact metrizable space $K$ and every $1 \leq p \leq \infty$.
\end{abstract}

\maketitle

\tableofcontents

\bigskip
\input{1-Introduction}
\bigskip
\input{2-1-Organization-and-notation}
\bigskip
\input{3_Preliminary_results}

\input{3-0-1-The_ell_1_case}
\bigskip
\input{3-1_l_infty}
\bigskip
\input{4_proof_main_theorems}
\bigskip

\noindent\textbf{Acknowledgements.} This paper forms part of the author's PhD research at Lancaster University, conducted under the supervision of Professor N. J. Laustsen. He acknowledges with thanks the funding from the EPSRC (grant number EP/W524438/1) that has supported his studies. \medskip

For the purpose of open access, the author has applied a Creative Commons Attribution (CC BY) licence to any Author Accepted Manuscript version arising. \medskip

\noindent\textbf{Data availability.} No data was used for the research described in the article.


\input{bibliography}
\end{document}

%% file: 1-Introduction.tex
\section{Introduction and organisation}

A Banach space $X$ is said to be \emph{primary} if it cannot be decomposed, in a complemented sense, into two essentially different pieces. More precisely, $X$ is primary if, whenever $P\colon X\to X$ is a projection, either $PX$ or $(I_X-P)X$ is isomorphic to $X$.

The primariness problem for the classical Banach spaces played an important role in the development of the isomorphic theory of Banach spaces during the 1960s and 1970s. The first major successes concerned the sequence spaces: Pe{\l}czy\'nski proved the stronger property of primeness for $c_0$ and $\ell_p$, $1\leq p<\infty$ \cite{Pelczynski1960}, and Lindenstrauss obtained the corresponding result for $\ell_\infty$ \cite{Lin1967}. This was followed by the primariness of $C[0,1]$, proved by Lindenstrauss and Pe{\l}czy\'nski \cite{LindenstraussPelczynski1971}, and of $L_p$, $1\leq p<\infty$, obtained from Enflo's ideas as presented by Maurey \cite{Maurey1975}. The spaces $C(\alpha)$ of continuous functions on ordinal intervals were then treated by Billard in the countable case \cite{Billard1978} and by Alspach and Benyamini in full generality \cite{AlspachBenyamini1977}. Taken together, these results provided a positive solution to the primariness problem for classical spaces and introduced powerful methods for analysing their geometry.

The primariness of many other Banach spaces has also been established over the years; we refer the reader to \cite[Section 2]{AcuavivaKania2026} for a more detailed account. A particularly important direction concerns bi-parameter spaces, that is, spaces obtained by combining two classical Banach spaces. In the sequence-space setting, the corresponding results for the spaces $\ell_p(\ell_q)$, $1\leq p,q\leq\infty$, were obtained through the work of Casazza--Kottman--Lin \cite{CKL1976} and Capon \cite{Capon1982PLMS}. These results naturally extend to multi-parameter sequence spaces; see the paper of the present author and Kania \cite{AcuavivaKania2026}.

Bi-parameter spaces involving $L_p$-spaces form another major part of this picture. In that direction, Capon developed a substantial analysis of mixed $L_p$-spaces, establishing the primariness of the spaces $L_r(L_s)$, $\ell_p(L_q)$ and $L_p(\ell_q)$ for $1<r,s<\infty$ and $1\leq p,q<\infty$ \cite{Capon1980, Capon1980b, Capon1982PLMS, Capon1983}. The corresponding $\ell_\infty$-sums required separate methods: Wark proved the primariness of $\ell_\infty(L_p)$ for $1<p<\infty$ \cite{Wark2007}, and M\"uller subsequently treated $\ell_\infty(L_1)$ \cite{Muller2012}. More recently, Lechner, Motakis, M\"uller and Schlumprecht proved the primariness of $L_1(L_p)$, $1 < p < \infty$ \cite{LMMSS2022}, and pointed out several remaining unresolved primariness problems for classical bi-parameter spaces.

Motivated by this broader programme, this paper studies bi-parameter spaces involving $C(K)$-spaces, a class for which the corresponding primariness problems have so far remained largely open. More precisely, we study $\ell_p$-sums of $C(K)$-spaces. Previously known primariness results for spaces of the form $\ell_p(X)$ typically relied on properties of a basis of $X$, perhaps with $\ell_p(L_1)$ being a notable exception. Such methods do not appear to be directly applicable in the $C(K)$-space setting. The approach taken here is instead based on the reproducibility properties of $C(K)$ itself, together with the structure of the underlying compact space $K$. This naturally brings the primary factorisation property into play. Recall that the (uniform) primary factorisation property, denoted by (U)PFP, asserts a dichotomy for operators on the space in terms of factorisation of the identity; see \Cref{def: UPFP} for details. The same point of view was used by the present author and Kania in the context of primariness of uncountable $\ell_p$-sums; see \cite[Section 4]{AcuavivaKania2026}.

Our main results are as follows.

\begin{theorem}\label{th: main-theorem}
    For every ordinal $\alpha$ and every $1 < p \leq \infty$, the space $\ell_p(C(\alpha))$ has the UPFP. In particular, these spaces are primary.
\end{theorem}

Using similar methods, we also obtain the corresponding result for the uncountable compact metric case.

\begin{theorem}\label{th: main-theorem2}
    For every $1 < p \leq \infty$, the space $\ell_p(C[0,1])$ has the UPFP.
    In particular, these spaces are primary.
\end{theorem}

Finally, we treat the case $p = 1$; in this case, we obtain a general criterion ensuring that the UPFP is transferred to the corresponding $\ell_1$-sum.

\begin{theorem}\label{th: main-theorem3}
    Let $X$ be a Banach space such that either $X\simeq c_0(X)$ or $X\simeq \ell_p(X)$ for some $1 < p \leq \infty$. Suppose moreover that $X$ contains no complemented copy of $\ell_1$. Then the UPFP passes from $X$ to $\ell_1(X)$; that is, if $X$ has the UPFP, then $\ell_1(X)$ has the UPFP. In particular, $\ell_1(X)$ is primary.
\end{theorem}

Combining these results with Milutin's theorem \cite{Milutin1966} and the Bessaga--Pe{\l}czy\'nski classification of the spaces $C(K)$ for countable compact metric spaces \cite{BessagaPelczynski1960}, we obtain a complete answer to one of the problems listed in \cite{LMMSS2022}, which motivated our interest in this question.

\begin{corollary}\label{cor: compact-metric-primary}
    Let $K$ be a compact metric space and let $1 \leq p \leq \infty$. Then $\ell_p(C(K))$ is primary.
\end{corollary}

The proofs of these theorems rely on reducing operators \break $T\colon \ell_p(C(K))\to \ell_p(C(K))$ to diagonal form. This is achieved in two steps: by reducing to upper triangular operators, and by reducing to lower triangular operators. Together, these reductions yield the desired diagonal form. At that point, one can factorise the identity by invoking the UPFP of the underlying space $C(K)$.  \\

Lastly, although our original motivation was to treat the case $C(K)$ where $K$ is metrizable, the methods developed here are not inherently limited to that setting. They appear to apply more broadly to other sums of the form $\ell_p(X)$, where $1\leq p\leq\infty$, especially when $X$ is a $C(K)$-space. For example, the same reductions to diagonal operators apply when $K$ is the double arrow space $K_{\mathrm{d.a.}}$. Thus, provided one can argue that $C(K_{\mathrm{d.a.}})$ has the UPFP, it would follow that $\ell_p(C(K_{\mathrm{d.a.}}))$ has the UPFP, and hence is primary, for every $1\leq p\leq\infty$. This may already be implicit in Michalak's proof of the primariness of $C(K_{\mathrm{d.a.}})$ \cite{Michalak2005}, but we have not checked the required uniform estimates.

\subsubsection*{Organisation.} We briefly describe the structure of the paper. In \Cref{sec: notation} we introduce the notation used throughout, and recall some background results concerning the topology of the ordinal interval $[0,\alpha]$ and the unit interval $[0,1]$, as well as operators on $C(\alpha)$ and $C[0,1]$. The main technical work is carried out in \Cref{sec: preliminary-lps,sec: ell1,sec: l-infty}. In \Cref{sec: preliminary-lps} we prove the upper and lower triangular reductions for the spaces $\ell_p(C(K))$, $1<p<\infty$, while in \Cref{sec: ell1} we treat the case $p=1$. The corresponding arguments for $\ell_\infty$-sums, where one must take additional care with operator matrices, are given in \Cref{sec: l-infty}. Finally, in \Cref{sec: main-theorems}, we use these technical reductions to prove \Cref{th: main-theorem,th: main-theorem2,th: main-theorem3}.

%% file: 2-1-Organization-and-notation.tex
\section{Notation and preliminary results}\label{sec: notation}

\subsection{Definitions and basic notation}

We follow standard notation and conventions unless stated otherwise. Throughout, the scalar field, denoted by $\mathbb{K}$, may be either $\mathbb{R}$ or $\mathbb{C}$. We reserve the term \emph{operator} for bounded linear maps between Banach spaces; thus, operators are always assumed to be bounded. For a Banach space $X$, we write $I_X \colon X \to X$ for the identity operator on $X$, $B_X$ for its unit ball and $\mathscr{B}(X)$ for the space of operators on $X$. More specialized notation will be introduced as needed.

\begin{definition}
    Let $T \colon X \to Y$ and $S \colon Z \to W$ be operators between Banach spaces $X$, $Y$, $Z$ and $W$. We say that $S$ \emph{factors through} $T$ with constant $C$ if there exist operators $V \colon Z \to X$ and $U \colon Y \to W$ such that
    \begin{equation*}
        UTV = S \quad \text{and} \quad \norm{U}\norm{V} \leq C.
    \end{equation*}
    If this holds for some constant $C \geq 1$, then we simply say that $S$ factors through $T$.
\end{definition}

\begin{definition}\label{def: UPFP}
    Let $X$ be a Banach space.
    \begin{enumerate}[label = (\alph*), ref = (\alph*)]    
        \item We say that $X$ has the \emph{primary factorisation property} (PFP) if for every bounded operator $T \colon X \to X$, the identity $I_X$ factors through either $T$ or $I_X - T$.
        \item Let $C \geq 1$. We say that $X$ has the \emph{$C$-primary factorisation property} ($C$-PFP) if, for every operator $T \colon X \to X$, the identity $I_X$ factors through either $T$ or $I_X - T$ with constant $C$.
        \item We say that $X$ has the \emph{uniform primary factorisation property} (UPFP) if there exists $C \geq 1$ such that $X$ has the $C$-PFP.
    \end{enumerate}
\end{definition}

\subsubsection{Sums of Banach spaces with respect to an unconditional basis.}

Let $E$ be a Banach space with a $1$-unconditional basis $(e_n)_{n \in \N}$, and let $(X_n)_{n \in \N}$ be a sequence of Banach spaces. We define the $E$-direct sum of these spaces by
\begin{equation*}
    \left(\bigoplus_{n \in \N} X_n \right)_{E}
    =
    \left\{ x = (x_n)_{n \in \N} \in \prod_{n \in \N} X_n : \sum_{n \in \N} \norm{x_n} e_n \in E \right\},
\end{equation*}
and equip it with the norm
\begin{equation*}
    \norm{(x_n)_{n \in \N}}
    =
    \left\| \sum_{n \in \N} \norm{x_n} e_n \right\|,
\end{equation*}
which makes it a Banach space. In the case where all the spaces $X_n$ are equal to a space $X$, we simply denote this space by $E(\N, X)$. In the particular case of $\ell_p$-spaces, we mostly omit the index set $\mathbb{N}$ and write $\ell_p(X)$ for $\ell_p(\N, X)$. Observe that this sum depends not only on the Banach space $E$, but also on the choice of basis $(e_n)_{n \in \N}$; however, this dependence will remain implicit throughout.

We consider now the case where $X_n = X$ for all $n \in \N$. For each $x = (x_n)_{n \in \N} \in E(\N, X)$, we define the \emph{support} of $x$ by
\begin{equation*}
    \supp(x) = \{n \in \N : x_n \neq 0\}.
\end{equation*}
For each $M \subseteq \N$, we define the closed subspace
\begin{equation*}
    E(M, X) = \{x \in E(\N, X) : \supp(x) \subseteq M\} \subseteq E(\N, X),
\end{equation*}
and let $P_M \colon E(\N, X) \to E(M, X)$ and $J_M \colon E(M, X) \to E(\N, X)$ denote the canonical projection and inclusion, respectively. Then, naturally, $P_M J_M = I_{E(M, X)}$. In the case where $M = \{m\}$ is a singleton, we simply write $P_m$ and $J_m$ in place of $P_{\{m\}}$ and $J_{\{m\}}$.

To each operator $T \colon E(\N, X) \to E(\N, X)$ we can associate a matrix of operators $(T_{n,m})_{n,m \in \N}$ in the natural way, namely by defining $T_{n,m} \colon X \to X$ by
\begin{equation*}
    T_{n,m} = P_n T J_m.
\end{equation*}
Then, for all $x = (x_n)_{n \in \N}$ and all $n \in \N$, we have
\begin{equation*}
    P_n T x = \sum_{m \in \N} T_{n,m} x_m.
\end{equation*}

Let $T \colon E(\N, X) \to E(\N, X)$ be an operator with matrix representation $(T_{n,m})_{n,m \in \N}$. We say that $T$ is \emph{upper triangular} if
\begin{equation*}
    T_{n,m} = 0 \qquad \text{whenever } n > m.
\end{equation*}
Similarly, we say that $T$ is \emph{lower triangular} if
\begin{equation*}
    T_{n,m} = 0 \qquad \text{whenever } n < m.
\end{equation*}
If $T$ is both upper and lower triangular, we say that $T$ is \emph{diagonal}.

Let $(T_n)_{n \in \N}$ be a sequence of operators $T_n \colon X \to X$ such that $\sup_{n \in \N} \norm{T_n} < \infty$. Then we may define the diagonal operator
\begin{equation*}
    T = \diag(T_n : n \in \N)
\end{equation*}
to be the operator $T \colon E(\N, X) \to E(\N, X)$ given by
\begin{equation*}
    P_m T x = T_m x_m
\end{equation*}
for all $m \in \N$ and all $x = (x_n)_{n \in \N}$. It is clear that $T$ is an operator and that
\begin{equation*}
    \norm{T} = \sup_{n \in \N} \norm{T_n}.
\end{equation*}

\subsubsection{The $\ell_\infty$ case.}

For a sequence $(X_n)_{n \in \N}$ of Banach spaces, one analogously defines the $\ell_\infty$-direct sum
\begin{equation*}
    \left(\bigoplus_{n \in \N} X_n \right)_{\ell_\infty}
    =
    \left\{ x = (x_n)_{n \in \N} \in \prod_{n \in \N} X_n : \sup_{n \in \N} \norm{x_n} < \infty \right\},
\end{equation*}
equipped with the norm
\begin{equation*}
    \norm{(x_n)_{n \in \N}} = \sup_{n \in \N} \norm{x_n}.
\end{equation*}
In the case where $X_n = X$ for all $n \in \N$, we simply write $\ell_\infty(X)$. We carry over most of the notation introduced above in the natural way. Thus, for instance, if $M \subseteq \N$, we still write $\ell_\infty(M, X)$ for the closed subspace consisting of those elements whose support is contained in $M$, and we denote by $P_M$ and $J_M$ the canonical projection and inclusion, respectively.

Moreover, for an operator $T \colon \ell_\infty(X) \to \ell_\infty(X)$, we may still associate to it an operator matrix $(T_{n,m})_{n,m \in \N}$ by setting
\begin{equation*}
    T_{n,m} = P_n T J_m.
\end{equation*}
However, in this setting, one must proceed with care, since the matrix does not, in general, determine the operator, as the following elementary example shows.

\begin{example}
    Let $X = \mathbb{K}$ and consider the operator $T\colon\ell_\infty(\mathbb{K})\to\ell_\infty(\mathbb{K})$ given by
    \begin{equation*}
        Tx=(\Lambda(x),0,0,\dots),
    \end{equation*}
    where $\Lambda$ is a Banach limit; see, for instance, \cite[Theorem~III.7.1]{Conway1990}. Thus $\Lambda\in\ell_\infty^*$, and $\Lambda$ agrees with the usual limit on convergent sequences. In particular, $\Lambda$ vanishes on $c_0$. Hence $T\neq 0$, while $T_{n,m} = P_n T J_m=0$ for all $n,m\in\N$, since $J_m \mathbb{K}\subseteq c_0$.
\end{example}

Thus, in contrast with the case of unconditional sums, the operator matrix does not in general suffice to describe the behaviour of an operator. This leads us to introduce the following stronger notion of lower triangularity.

\begin{definition}\label{def: lower-triangular-special}
    Let $T \colon \ell_\infty(X) \to \ell_\infty(X)$ be an operator. We say that $T$ is \emph{lower triangular} if
    \begin{equation*}
        T_{n,m} = 0 \qquad \text{whenever } n < m,
    \end{equation*}
    and, moreover, $T$ acts according to its matrix, in the sense that
    \begin{equation*}
        P_n T x = \sum_{m \leq n} T_{n,m} x_m
    \end{equation*}
    for all $n \in \N$ and all $x = (x_m)_{m \in \N} \in \ell_\infty(X)$.
\end{definition}

Similarly, we strengthen the notion of a diagonal operator.

\begin{definition}
    Let $T \colon \ell_\infty(X) \to \ell_\infty(X)$ be an operator. We say that $T$ is \emph{diagonal} if
    \begin{equation*}
        T_{n,m} = 0 \qquad \text{whenever } n \neq m,
    \end{equation*}
    and, moreover, $T$ acts according to its matrix, in the sense that
    \begin{equation*}
        P_m T x = T_{m,m} x_m
    \end{equation*}
    for all $m \in \N$ and all $x = (x_n)_{n \in \N} \in \ell_\infty(X)$.
\end{definition}

Naturally, given a uniformly bounded sequence of operators $T_{n}: X \to X$, we can still define the associated diagonal operator $T = \diag (T_n: n \in \N)$.

\subsection{Background and topological results in \texorpdfstring{$C(\alpha)$}{C(alpha)}}

We recall some standard topological notions, following the notation and presentation of Alspach and Benyamini \cite{AlspachBenyamini1977}. If $A$ is a subset of a topological space, we write $A^{(1)}$ for the set of accumulation points of $A$. The higher derived sets are then defined recursively by setting $A^{(\alpha+1)} = \bigl(A^{(\alpha)}\bigr)^{(1)}$ for each ordinal $\alpha$, and $A^{(\alpha)} = \bigcap_{\beta<\alpha} A^{(\beta)}$ whenever $\alpha$ is a limit ordinal. If, for some ordinal $\alpha$, the set $A^{(\alpha)}$ is finite of cardinality $n$, then $(\alpha,n)$ is called the \emph{characteristic system} of $A$.

Now let $A$ be a compact subset of an ordinal space. In this case, the subspace topology on $A$ coincides with the order topology inherited by $A$. Consequently, there is an ordinal $\alpha$ such that $A^{(\alpha)}$ is finite. We denote this $\alpha$ by $\tau(A)$ and refer to it as the \emph{type} of $A$. Moreover, two compact subsets of ordinal spaces are homeomorphic if and only if they have the same characteristic system; see, for instance, \cite{Baker1972}. In particular, we shall use the following classical consequence.

\begin{theorem}\label{th: Baker}
    Let $A$ be a compact subset of an ordinal space with characteristic system $(\eta,n)$, where $\eta$ is an ordinal and $n \in \N$. Then $A$ is homeomorphic to $[0,\omega^\eta \cdot n]$.
\end{theorem}

We also record the following elementary fact.

\begin{lemma}[Finite indivisibility of closed subsets of ordinal intervals]\label{lmm: indivisibility-compact-subsets}
    Let $A$ be a closed subset of an ordinal space with characteristic system $(\eta,1)$ for some ordinal $\eta$. Suppose that $A_1,\dots,A_N \subseteq A$ are closed subsets such that $A = \bigcup_{n=1}^N A_n$. Then there exists $1 \leq n_0 \leq N$ such that $A_{n_0}$ is homeomorphic to $A$.
\end{lemma}
\begin{proof}
    Since $A$ has characteristic system $(\eta,1)$, the set $A^{(\eta)}$ consists of exactly one point, say $A^{(\eta)} = \{a\}$. Since derived sets commute with finite unions of closed sets, we have
    \begin{equation*}
        \{a\} = A^{(\eta)} = \bigcup_{n=1}^N A_n^{(\eta)}.
    \end{equation*}
    It follows that there exists $1 \leq n_0 \leq N$ such that $A_{n_0}^{(\eta)} = \{a\}$. Therefore, $A_{n_0}$ has characteristic system $(\eta,1)$, and hence $A_{n_0}$ is homeomorphic to $A$ by \Cref{th: Baker}.
\end{proof}

We will need the following definition.

\begin{definition}
    Let $K$ be a compact Hausdorff space and let $A\subseteq K$ be closed. Denote by $R_A\colon C(K)\to C(A)$ the restriction operator. We say that an operator $S_A\colon C(A)\to C(K)$ is a \emph{simultaneous extension operator} if
\begin{equation*}
    R_A S_A = I_{C(A)}.
\end{equation*}
\end{definition}

We now introduce several facts concerning operators acting on $C(\alpha)$, as proven by Alspach and Benyamini. We start with the existence of a simultaneous extension operator \cite[Lemma 1.1 (c)]{AlspachBenyamini1977}.

\begin{proposition}[Simultaneous extension operator]\label{prop: simultaneous-extension}
    Let $\alpha$ be an ordinal, and let $A$ be a closed subset of $[0,\alpha]$. Then there exists a norm-one simultaneous extension operator
    \begin{equation*}
        S_A \colon C(A) \to C(\alpha).
    \end{equation*}
\end{proposition}

We shall also need the following consequence of the main result of Alspach and Benyamini \cite[Theorem 1]{AlspachBenyamini1977}.

\begin{theorem}[Alspach--Benyamini]\label{th: alspach-benyamini}
    Let $\alpha$ be an ordinal. Suppose that $C(\alpha)$ is not isomorphic to $C(\xi \cdot n)$ for any uncountable regular ordinal $\xi$ and any $n \in \mathbb{N}$ with $n \geq 2$. Then $C(\alpha)$ has the UPFP.
\end{theorem}

The previous theorem is not stated explicitly in this form in \cite{AlspachBenyamini1977}, but it can be extracted directly from the proof of their main result; see \cite[Proposition 1.9]{Acuaviva2025} for a detailed exposition. Moreover, the estimates in each step can be made quantitative. Since the UPFP is preserved under isomorphisms \cite[Proposition 2.5]{AcuavivaKania2026}, this yields the stated UPFP conclusion for the ordinal spaces covered by \Cref{th: alspach-benyamini}.

\subsection{Background and topological results in \texorpdfstring{$C[0,1]$}{C[0,1]}}

Similarly, we recall some background results for the unit interval $[0,1]$. We have the following formulation for topological indivisibility in this case.

\begin{lemma}[Countable topological indivisibility of $\lbrack 0,1\rbrack$] \label{lmm: countable-top-indiv}
    Let $(A_n)_{n=1}^\infty$ be a sequence of closed subsets of $[0,1]$ such that
    \begin{equation*}
        [0,1] = \bigcup_{n \in \N} A_n.
    \end{equation*}
    Then there exist $n_0 \in \N$ and a closed subset $A \subseteq A_{n_0}$
    such that $A$ is homeomorphic to $[0,1]$.
\end{lemma}

\begin{proof}
    Since $[0,1]$ is compact and metrizable, it is a Baire space. Hence, it
    cannot be written as a countable union of closed subsets with empty
    interior. Therefore there exists $n_0 \in \N$ such that $A_{n_0}$ has
    non-empty interior in $[0,1]$.

    Choose a non-empty relatively open subset $U$ of $[0,1]$ such that $U \subseteq A_{n_0}$. Since $U$ is non-empty and relatively open in $[0,1]$, it contains a non-degenerate closed interval $A=[a,b]$, with $a<b$. Then $A$ is homeomorphic to $[0,1]$, and $A \subseteq A_{n_0}$.
\end{proof}

We also record the following special case of the Borsuk--Dugundji extension theorem, see \cite[Theorem 5.1]{Dugundji1951}, which is the only form needed in what follows.

\begin{theorem}[Borsuk--Dugundji extension theorem]\label{thm: borsuk-dugundji}
    Let $A$ be a closed subset of $[0,1]$. Then there exists a norm-one simultaneous extension operator $S_A \colon C(A) \to C[0,1]$.
\end{theorem}

\begin{remark}
    Throughout this paper, $A$ will be a closed interval, so one does not need the full strength of the Borsuk--Dugundji extension theorem. Indeed, if $A=[a,b] \subseteq [0,1]$, a norm-one extension operator $S_A \colon C(A) \to C[0,1]$ can be obtained by linear interpolation.
\end{remark}

Finally, we record that $C[0,1]$ has the UPFP. The quantitative form needed for the UPFP can be extracted from Weis' proof \cite{Weis1986} of a theorem of Rosenthal \cite{Rosenthal1972}. The same conclusion may already be implicit in Rosenthal's original argument, although we have not verified this.

\begin{proposition}\label{prop: C01-UPFP}
    The space $C[0,1]$ has the UPFP.
\end{proposition}

\begin{proof}
    For an operator $T\in \mathscr{B}(C[0,1])$, denote by $\omega_T(t)$ the oscillation of the representing kernel $(T^*\delta_t)_{t\in[0,1]}$, as in \cite{Weis1986}. We shall use that the oscillation is subadditive: if $T,S\in\mathscr{B}(C[0,1])$, then
    \begin{equation*}
        \omega_{T+S}(t)\leq \omega_T(t)+\omega_S(t)
    \end{equation*}
    for every $t\in[0,1]$. Thus, fixing an operator $T \in \mathscr{B}(C[0,1])$, we show that it satisfies the dichotomy for the UPFP.

    Since the representing kernel of $I_{C[0,1]}$ is $t\mapsto \delta_t$, and since $\norm{\delta_s-\delta_t}=2$ whenever $s\neq t$, we have $\omega_{I_{C[0,1]}}(t)=2$ for every $t\in[0,1]$. Hence
    \begin{equation*}
        2=\omega_{I_{C[0,1]}}(t)\leq \omega_T(t)+\omega_{I_{C[0,1]}-T}(t)
    \end{equation*}
    for every $t\in[0,1]$. The sets
    \begin{equation*}
        A=\{t\in[0,1]:\omega_T(t)\geq 1\},\qquad B=\{t\in[0,1]:\omega_{I_{C[0,1]}-T}(t)\geq 1\}
    \end{equation*}
    are closed since $\omega_T$ and $\omega_{I_{C[0,1]}-T}$ are upper semi-continuous, and they cover $[0,1]$ by the previous inequality. By the Baire category theorem, one of them has non-empty interior. Suppose first that $A$ has non-empty interior. Since  $\omega_T$ is an upper semi-continuous function, the set of its continuity points is a dense $G_\delta$-set, so that we may choose $t_0\in \operatorname{int}(A)$ at which $\omega_T$ is continuous. This implies, in particular, that $\omega_T(t_0)\geq 1$.

    We now follow the proof of \cite[Theorem 2]{Weis1986}. Denote by $\Delta$ the Cantor set. The estimates in that proof are quantitative: if $\omega_T(t_0)$ is bounded below at a point of continuity of $\omega_T$, then the identity on $C(\Delta)$ factors through $T$ with a constant depending only on this lower bound. Applying this and using $\omega_T(t_0)\geq 1$, we obtain operators $U$ and $V$ such that
    \begin{equation*}
        UTV=I_{C(\Delta)}
    \end{equation*}
    and $\norm{U}\norm{V}\leq C$, where $C$ is independent of $T$.

    We observe that, although Weis writes the proof for the real scalar field, the same argument applies over the complex scalar field. The measure-theoretic part uses only the Lebesgue decomposition of finite Radon measures, total variation and mutual singularity. At the point where positive and negative parts of signed measures are used, one instead applies the Jordan decomposition to the real and imaginary parts of the relevant complex measures. This changes, at most, the absolute constant.
    
    Since $C(\Delta)$ is isomorphic to $C[0,1]$, it follows that $I_{C[0,1]}$ factors through $T$ with constant bounded independently of $T$. If instead $B$ has non-empty interior, the same argument applied to $I_{C[0,1]}-T$ gives the corresponding factorisation through $I_{C[0,1]}-T$. Therefore, $C[0,1]$ has the UPFP.
\end{proof}

%% file: 3_Preliminary_results.tex
\section{Results for matrix operators}\label{sec: preliminary-lps}

\subsection{Reduction to upper triangular form.}

\subsubsection{The $C(\alpha)$ case} We begin by proving that, under suitable structural assumptions, one may pass to a subcopy of the basis on which the operator $T$ has upper triangular form. Since this argument works, without additional work, in a fairly general setting, we state the result at that level of generality. The underlying ideas are essentially those already present in \cite{Laustsen2001}, where the interplay between operators $S \colon X \to E$ and operators on the $E$-unconditional sum of $X$ is studied. Before stating the result, we require the following lemma.

\begin{lemma}\label{lmm: preliminary-upper-triangular}
    Let $E$ be a Banach space with a $1$-unconditional basis $(e_n)_{n \in \N}$, and let $X$ be a Banach space. Suppose that every operator $S \colon X \to E$ is compact. Then, for every operator
    \begin{equation*}
        T \colon X \to E(\N, X),
        \qquad
        T = (T_n)_{n \in \N},
    \end{equation*}
    we have 
    \begin{equation*}
        \lim_{n \to \infty} \norm{T_n} = 0.
    \end{equation*}
\end{lemma}

\begin{proof}
    Suppose, towards a contradiction, that the conclusion fails. Then there exist an operator
    \begin{equation*}
        T \colon X \to E(\N, X),
        \qquad
        T = (T_n)_{n \in \N},
    \end{equation*}
    and some $\varepsilon > 0$ such that, for every $N \in \N$, there exists $n > N$ with $\norm{T_n} \geq \varepsilon$.

    Hence, we may choose a strictly increasing sequence $(n_k)_{k \in \N}$ in $\N$ such that
    \begin{equation*}
        \norm{T_{n_k}} \geq \varepsilon
        \qquad \text{for all } k \in \N.
    \end{equation*}
    For each $k \in \N$, choose $x_k \in B_X$ such that $\norm{T_{n_k}x_k} > \varepsilon/2$.
    By the Hahn--Banach theorem, for each $k \in \N$ there exists $x_k^* \in B_{X^*}$ such that $ x_k^*(T_{n_k}x_k) > \varepsilon/2.$

    Define an operator $D \colon E(\N, X) \to E$ by
    \begin{equation*}
        D\bigl((x_n)_{n \in \N}\bigr) = \sum_{k=1}^\infty x_k^*(x_{n_k}) e_{n_k},
        \qquad (x_n)_{n \in \N} \in E(\N, X).
    \end{equation*}
    Since the basis $(e_n)_{n \in \N}$ is $1$-unconditional and $\norm{x_k^*} \leq 1$ for all $k \in \N$, it follows that $D$ is a bounded linear operator. Set $\widehat{T} = D \circ T \colon X \to E$.
    By assumption, the operator $\widehat{T}$ is compact, so the sequence $(\widehat {T} x_k)_{k \in \N}$ has a norm-convergent subsequence. Passing to a subsequence, we may suppose that $(\widehat T x_k)_{k \in \N}$ converges in norm to some $y \in E$.
    
    For each $k \in \N$, from the construction we have
        \begin{equation*}
            e_{n_k}^*(\widehat T x_k) = x_k^* (T_{n_k} x_k)
            >
            \frac{\varepsilon}{2}.
    \end{equation*}
    Since $\| e_{n_k}^* \|=1$ for all $k \in \N$, it follows that
    \begin{equation*}
            \bigl| e_{n_k}^*(y) \bigr|
            \geq
            \bigl| e_{n_k}^*(\widehat T x_k) \bigr|
            -
             \| \widehat T x_k-y \|
            >
            \frac{\varepsilon}{2} - \| \widehat T x_k-y \|.
        \end{equation*}
    Since $\widehat T x_k \to y$ in norm, it follows that for all sufficiently large $k$ we obtain $| e_{n_k}^*(y) | > \varepsilon/4$. This is impossible, because the coefficients of $y$ with respect to the Schauder basis $(e_n)_{n\in\N}$ must tend to zero. This contradiction completes the proof.
\end{proof}

As an immediate consequence, we obtain the following.

\begin{corollary}[Strong forward reduction]\label{cor: preliminary-upper-triangular}
    Let $E$ be a Banach space with a $1$-unconditional basis $(e_n)_{n \in \N}$, and let $X$ be a Banach space. Suppose that every operator $S \colon X \to E$ is compact. Then, for every operator
    \begin{equation*}
        T \colon X \to E(\N, X),
        \qquad
        T = (T_n)_{n \in \N},
    \end{equation*}
    and every $\varepsilon > 0$, there exists an infinite subset $M \subseteq \N$ such that
    \begin{equation*}
        \sum_{n \in M} \norm{T_n} < \varepsilon.
    \end{equation*}
\end{corollary}

We now obtain the following.

\begin{lemma}\label{lmm: lemma-upper-triangular-1}
    Let $E$ be a Banach space with a $1$-unconditional basis $(e_n)_{n \in \N}$, and let $X$ be a Banach space. Suppose that every operator $S \colon X \to E$ is compact. Then, for every operator
    \begin{equation*}
        T \colon E(\N, X) \to E(\N, X),
        \qquad
        T = (T_{n,m})_{n,m \in \N},
    \end{equation*}
    and every $\varepsilon > 0$, there exists an infinite subset $M = \{m_j: j \in \N \} \subseteq \N$ such that
    \begin{equation*}
        \sum_{\substack{i > j}} \norm{T_{m_i,m_j}} < \varepsilon/2^{j}
        \qquad \text{for all } j \in \N.
    \end{equation*}
\end{lemma}

\begin{proof}
    We argue by recursively applying \Cref{cor: preliminary-upper-triangular} to the columns of $T$. Choose $m_1 = 1$. Applying \Cref{cor: preliminary-upper-triangular} to the $m_1$th column of $T$, i.e. to $TJ_{m_1}$, we obtain an infinite subset $M_1 \subseteq \N$ such that
    \begin{equation*}
        \sum_{\substack{i \in M_1 \\ i > m_1}} \norm{T_{i,m_1}} < \varepsilon/2.
    \end{equation*}

    Suppose now that we have constructed $1 = m_1 < \dots < m_k$ and infinite sets
    \begin{equation*}
        \N = M_0 \supseteq M_1 \supseteq \dots \supseteq M_k
    \end{equation*}
    such that $m_j \in M_{j-1}$ and
    \begin{equation*}
        \sum_{\substack{i \in M_j \\ i > m_j}} \norm{T_{i,m_j}} < \varepsilon/2^{j}
    \end{equation*}
    for each $j \leq k$. Choose any $m_{k+1} \in M_k$ with $m_{k+1} > m_k$. Applying \Cref{cor: preliminary-upper-triangular} to the $m_{k+1}$th column of $T$, restricted to the rows indexed by $M_k$, we obtain an infinite set $M_{k+1} \subseteq M_k$ such that
    \begin{equation*}
        \sum_{\substack{i \in M_{k+1} \\ i > m_{k+1}}} \norm{T_{i,m_{k+1}}} < \varepsilon/2^{k+1}.
    \end{equation*}
    This completes the inductive construction. Set
    \begin{equation*}
        M = \{m_k : k \in \N\}.
    \end{equation*}
    Since $M \setminus \{m_1,\dots,m_j\} \subseteq M_j$ for each $j \in \N$, it follows that
    \begin{equation*}
        \sum_{i > j} \norm{T_{m_i,m_j}}
        \leq
        \sum_{\substack{i \in M_j \\ i > m_j}} \norm{T_{i,m_j}}
        < \varepsilon/2^{j}
    \end{equation*}
    for every $j \in \N$. This proves the result.
\end{proof}

\begin{remark}
    If one assumes instead that every operator $S \colon E \to X$ is compact, then one obtains a similar result, namely, a reduction to lower triangular form, by following the ideas developed by Laustsen in \cite{Laustsen2001}. Since this will not be needed for our purposes, we do not make this explicit here.
\end{remark}

We now get the claimed upper triangular form.

\begin{proposition}[Reduction to upper-triangular operators]\label{prop: prop-upper-triangular}
    Let $E$ be a Banach space with a $1$-unconditional basis $(e_n)_{n \in \N}$, and let $X$ be a Banach space. Suppose that every operator $S \colon X \to E$ is compact. Then, for every operator
    \begin{equation*}
        T \colon E(\N, X) \to E(\N, X),
        \qquad
        T = (T_{n,m})_{n,m \in \N},
    \end{equation*}
    and every $\varepsilon > 0$, there exists an infinite subset $M = \{m_j : j \in \N\} \subseteq \N$ such that
    \begin{equation*}
        \norm{P_M T J_M - U} < \varepsilon,
    \end{equation*}
    where $U \colon E(M, X) \to E(M, X)$ is the upper triangular operator defined by
    \begin{equation*}
        U_{m_i,m_j} =
        \begin{cases}
            T_{m_i,m_j}, & \text{if } i \leq j,\\
            0, & \text{if } i > j.
        \end{cases}
    \end{equation*}
\end{proposition}
\begin{proof}
    By \Cref{lmm: lemma-upper-triangular-1}, we may choose an infinite subset $M = \{m_j : j \in \N\} \subseteq \N$ such that
    \begin{equation*}
        \sum_{i > j} \norm{T_{m_i,m_j}} < \varepsilon/2^j
        \qquad \text{for all } j \in \N.
    \end{equation*}
    
    Define $U \colon E(M, X) \to E(M, X)$ as in the statement. It is clear that $U$ defines a linear map, although it need not, a priori, be bounded. The proof that $\norm{P_M T J_M - U} < \varepsilon$ shows in particular that $U$ is bounded, with norm at most $\norm{T} + \varepsilon$ and hence an operator.
    
    Let $x = (x_{m_j})_{j \in \N} \in E(M,X)$ with $\norm{x} \leq 1$. In particular, by $1$-unconditionality, we have $\norm{x_{m_j}} \leq 1$ for all $j \in \N$. It follows that
    \begin{equation*}
        \norm{((P_M T J_M - U)x)_{m_i}} \leq \sum_{j < i} \norm{T_{m_i,m_j}} \norm{x_{m_j}} \leq \sum_{j < i} \norm{T_{m_i,m_j}},
    \end{equation*}
    and hence, again by $1$-unconditionality,
    \begin{equation*}
        \norm{(P_M T J_M - U)x}
        \leq
        \sum_{i=1}^{\infty} \sum_{j < i} \norm{T_{m_i,m_j}}
        =
        \sum_{j=1}^{\infty} \sum_{i > j} \norm{T_{m_i,m_j}}
        <
        \sum_{j=1}^{\infty} \varepsilon/2^j
        =
        \varepsilon,
    \end{equation*}
    where the equality follows by re-indexing the non-negative double series. Therefore, $\norm{P_M T J_M - U} < \varepsilon$.
\end{proof}

Finally, we note that these results apply to $\ell_p$-sums of $C(\alpha)$ spaces whenever $1 \leq p < \infty$. The following is well known; we include a proof for completeness.

\begin{proposition}\label{prop: alpha-to-ellp-is-compact}
    Let $\alpha$ be an ordinal and let $1 \leq p < \infty$. Then every operator $S \colon C(\alpha) \to \ell_p$ is compact.
\end{proposition}
\begin{proof}
  By a classical theorem of Pe{\l}czy\'nski \cite{pelczynski1965strictly}, every operator from a $C(K)$-space into a Banach space either fixes a copy of $c_0$ or is weakly compact. Since $C(\alpha)^* = \ell_1(\alpha)$ and $\ell_1(\alpha)$ has the Schur property, every weakly compact operator from $C(\alpha)$ is compact. It follows that every operator $S \colon C(\alpha) \to \ell_p$ either fixes a copy of $c_0$ or is compact. By Pitt's theorem, every operator from $c_0$ to $\ell_p$ is compact, and thus no operator $S \colon C(\alpha) \to \ell_p$ can fix a copy of $c_0$. Consequently, every operator from $C(\alpha)$ to $\ell_p$ is compact.
\end{proof}

It is worth noting that these arguments also apply to $\ell_p$-sums of $C[0,1]$ for $1 \leq p < 2$.

\begin{proposition}\label{prop: compact-fromC01}
    Let $1 \leq p < 2$. Then every operator $S\colon C[0,1] \to \ell_p$ is compact.
\end{proposition}

\begin{proof}
    This is a standard consequence of classical factorisation theorems. Indeed, $C[0,1]^*$ is an $L_1$-space, and hence has cotype $2$. Moreover, $\ell_p$ has cotype $2$ for $1 \leq p \leq 2$, while $C[0,1]$ has the approximation property. Therefore, by Pisier's abstract Grothendieck theorem, see for instance \cite[Theorem 8.1.8]{albiac2006topics}, every operator from $C[0,1]$ to $\ell_p$ factors through a Hilbert space. Since every operator from a Hilbert space into $\ell_p$, $1\leq p<2$, is compact by Pitt's theorem, the result follows.
\end{proof}

The previous proposition cannot be extended to $p\geq 2$, since in that range there are non-compact operators from $C[0,1]$ to $\ell_p$.

\subsubsection{The $C[0,1]$ case.}\label{subsec: C01forward} We now explain how to obtain the upper triangular reduction in this setting. The key additional ingredient is the isomorphism $c_0(C[0,1]) \simeq C[0,1]$, which provides additional room for the gliding-hump argument along $c_0$. Readers familiar with the proof of primariness of $L_p$ by Alspach, Enflo and Odell \cite{AEO1977} may recognise a related device: there, the isomorphism $L_p \simeq \ell_2(L_p)$, together with additional properties of the Haar basis, is used to create room along auxiliary $\ell_2$-coordinates.

\begin{lemma}[Forward reduction]\label{lmm:c0-lp-coordinate-smallness}
    Let $1\leq p<\infty$, let $X$ and $Y$ be Banach spaces and let $T\colon c_0(X) \to \ell_p(Y)$ be an operator. Then, for every $\varepsilon>0$, there exist $n\in\N$ and an infinite set $M\subseteq\N$ such that $\norm{P_M T J_n}<\varepsilon$.
\end{lemma}

\begin{proof}
    We proceed by contradiction and assume that the statement is false. Then there is $\varepsilon>0$ such that, for every $n\in\N$ and every infinite set $M\subseteq\N$, we have $\norm{P_MTJ_n}\geq 2\varepsilon$. Choose $r\in\N$ such that $r\varepsilon^p>\norm{T}^p$, and partition $\N$ into infinite pairwise disjoint sets $M_1,\dots,M_r$.

    By assumption, for each $s=1,\dots,r$ we have $\norm{P_{M_s}TJ_s}\geq 2\varepsilon$. Hence, we may choose $x_s\in B_X$ and $x^*_s\in B_{\ell_p(M_s,Y)^*}$ such that $|x^*_sP_{M_s}TJ_sx_s|>\varepsilon$.

    For each choice of signs $\theta = (\theta_1,\dots,\theta_r ) \in \{-1,1\}^r$, set
    \begin{equation*}
        x_\theta=\sum_{s=1}^r \theta_s J_sx_s.
    \end{equation*}
    Since the vectors $J_sx_s$ have disjoint supports in $c_0(X)$, we have $\norm{x_\theta}\leq 1$. Thus $\norm{Tx_\theta}^p\leq \norm{T}^p$ for every choice of signs, and averaging gives
    \begin{equation*}
        \norm{T}^p \geq \mathbb E_\theta [\norm{Tx_\theta}^p].
    \end{equation*}
    Since the sets $M_1,\dots,M_r$ form a partition of $\N$, we have
    \begin{equation*}
        \mathbb E_\theta [\norm{Tx_\theta}^p] = \sum_{s=1}^r \mathbb E_\theta [ \norm{P_{M_s}Tx_\theta}^p] \geq \sum_{s=1}^r \mathbb E_\theta [|x_s^*P_{M_s}Tx_\theta|^p].
    \end{equation*}
    For each fixed $s$, we have
    \begin{equation*}
        \mathbb E_\theta [|x_s^*P_{M_s}Tx_\theta|^p] = \mathbb E_\theta [|\theta_s x_s^*P_{M_s}Tx_\theta|^p] \geq \left|\mathbb E_\theta [\theta_s x_s^*P_{M_s}Tx_\theta ]\right|^p,
    \end{equation*}
    where the first equality follows from $|\theta_s|=1$ and the inequality from Jensen's inequality applied to the convex function $z\mapsto |z|^p$. Moreover,
    \begin{equation*}
        \mathbb E_\theta [\theta_s x_s^*P_{M_s}Tx_\theta] = \sum_{t=1}^r \mathbb E_\theta[\theta_s\theta_t]x_s^*P_{M_s}TJ_tx_t = x_s^*P_{M_s}TJ_sx_s,
    \end{equation*}
    because $\mathbb E_\theta[\theta_s\theta_t]=0$ for $s\neq t$, while $\mathbb E_\theta[\theta_s^2]=1$. Hence
    \begin{equation*}
        \mathbb E_\theta [|x_s^*P_{M_s}Tx_\theta|^p] \geq |x_s^*P_{M_s}TJ_sx_s|^p > \varepsilon^p.
    \end{equation*}
    Therefore
    \begin{equation*}
        \norm{T}^p \geq \mathbb E_\theta [\norm{Tx_\theta}^p] > r\varepsilon^p,
    \end{equation*}
    contradicting the choice of $r$. This contradiction proves the result.
\end{proof}

\begin{remark}\label{rem: upper-triangular-form}
    Observe that the previous result also applies to $E$-sums, that is, to operators $T\colon c_0(X)\to E(Y)$, provided that $E$ has an unconditional basis and no block sequence of the basis of $E$ spans a copy of $c_0$. Indeed, the same proof applies after splitting $\N$ into countably many infinite sets $(M_s)_{s\in\N}$; failure of the conclusion would produce, after truncating and passing to a subsequence, a block sequence equivalent to the canonical basis of $c_0$.
\end{remark}

We shall use the following notation. For each $n\in\N$, we denote by
\begin{equation*}
    \widehat{J}_n\colon c_0(X)\to \ell_p(c_0(X))
\end{equation*}
the canonical isometric embedding into the $n$-th outer coordinate. For each pair $(n,m)\in\N^2$, we denote by
\begin{equation*}
    J_{n,m}=\widehat{J}_nJ_m\colon X\to \ell_p(c_0(X))
\end{equation*}
that is, the canonical isometric embedding into the coordinate $(n,m)$, where we recall that $J_m\colon X\to c_0(X)$ is the canonical embedding into the $m$-th coordinate.

More generally, given two sequences $\mathbf{n}=(n_j)_{j\in\N}$ and $\mathbf{m}=(m_j)_{j\in\N}$ in $\N$, with $\mathbf{n}$ increasing, we denote by
\begin{equation*}
    J_{\mathbf{n},\mathbf{m}}\colon \ell_p(X)\to \ell_p(c_0(X))
\end{equation*}
the isometric embedding which places the $j$-th coordinate of $\ell_p(X)$ in the coordinate $(n_j,m_j)$ of $\ell_p(c_0(X))$. We denote by
\begin{equation*}
    R_{\mathbf{n},\mathbf{m}}\colon \ell_p(c_0(X))\to \ell_p(X)
\end{equation*}
the corresponding coordinate projection. Thus, if $y=(y_{n,m})_{n,m\in\N}\in \ell_p(c_0(X))$, then
\begin{equation*}
    R_{\mathbf{n},\mathbf{m}}y=(y_{n_j,m_j})_{j\in\N}.
\end{equation*}
Naturally, $R_{\mathbf{n},\mathbf{m}}J_{\mathbf{n},\mathbf{m}}=I_{\ell_p(X)}$. For $M\subseteq\N$, we write $\widehat{P}_M\colon \ell_p(c_0(X))\to \ell_p(c_0(X))$ for the projection onto the outer coordinates belonging to $M$. 

\begin{lemma}\label{lmm:lp-c0-upper-reduction}
    Let $1\leq p<\infty$, let $X$ be a Banach space, and let
    \begin{equation*}
        T\colon \ell_p(c_0(X))\to \ell_p(c_0(X))
    \end{equation*}
    be an operator. Then, for every $\varepsilon>0$, there exist increasing sequences $\mathbf{n}=(n_j)_{j\in\N}$ and $\mathbf{m}=(m_j)_{j\in\N}$ in $\N$, and an upper triangular operator $U\colon \ell_p(X)\to \ell_p(X)$, such that
    \begin{equation*}
        \norm{R_{\mathbf{n},\mathbf{m}}TJ_{\mathbf{n},\mathbf{m}}-U}<\varepsilon.
    \end{equation*}
\end{lemma}

\begin{proof}
    We shall recursively construct pairs $(n_j,m_j)\in\N^2$ and infinite subsets
    \begin{equation*}
        \N=M_0\supseteq M_1\supseteq M_2\supseteq \ldots
    \end{equation*}
    such that, for every $j\in\N$,
    \begin{align*}
        n_j\in M_{j-1}, \qquad M_j\subseteq M_{j-1}\cap\{n\in\N:n>n_j\}, \\
         m_j>m_{j-1}, \qquad  \text{and} \qquad \norm{\widehat{P}_{M_j}TJ_{n_j,m_j}}<\varepsilon / 2^{j+1},
    \end{align*}
    where we set $m_0=0$.

    Suppose that $M_{j-1}$ has been chosen. Choose $n_j\in M_{j-1}$, and set
    \begin{equation*}
        N_j=M_{j-1}\cap\{n\in\N:n>n_j\}.
    \end{equation*}
    Let $W_j\colon c_0(X)\to c_0(X)$ be the isometric embedding which places the $k$-th coordinate of $c_0(X)$ in the coordinate $m_{j-1}+k$. Applying \Cref{lmm:c0-lp-coordinate-smallness} to the operator
    \begin{equation*}
        \widehat{P}_{N_j}T\widehat{J}_{n_j}W_j\colon c_0(X)\to \ell_p(N_j,c_0(X)),
    \end{equation*}
    we obtain $k_j\in\N$ and an infinite subset $M_j\subseteq N_j$ such that
    \begin{equation*}
        \norm{\widehat{P}_{M_j}T\widehat{J}_{n_j}W_jJ_{k_j}}<\varepsilon / 2^{j+1}.
    \end{equation*}
    Set $m_j=m_{j-1}+k_j$. Since $W_jJ_{k_j}=J_{m_j}$, we have $\widehat{J}_{n_j}W_jJ_{k_j}=J_{n_j,m_j}$. Hence
    \begin{equation*}
        \norm{\widehat{P}_{M_j}TJ_{n_j,m_j}}<\varepsilon / 2^{j+1}.
    \end{equation*}
    This completes the recursive construction. In particular, both $(n_j)_{j\in\N}$ and $(m_j)_{j\in\N}$ are increasing.

    Set $S = R_{\mathbf{n},\mathbf{m}}TJ_{\mathbf{n},\mathbf{m}}$ and write $S = (S_{i,j})_{i,j\in\N}$ with respect to the natural decomposition of $\ell_p(X)$. For each $j\in\N$, define $H_j\colon X\to \ell_p(X)$ coordinate-wise by
    \begin{equation*}
        P_iH_jx=
        \begin{cases}
            S_{i,j}x, & i>j,\\
            0, & i\leq j.
        \end{cases}
    \end{equation*}
    If $i>j$, then $n_i\in M_j$. Hence, for every $j\in\N$ and every $x\in X$, we have
    \begin{equation*}
        \norm{H_jx}_{\ell_p(X)}=\left(\sum_{i>j}\norm{S_{i,j}x}^p\right)^{1/p}\leq \norm{\widehat{P}_{M_j}TJ_{n_j,m_j}x}<(\varepsilon/2^{j+1})\norm{x}.
    \end{equation*}
    In particular, $\norm{H_j}\leq \varepsilon/2^{j+1}$. Define $L\colon \ell_p(X)\to \ell_p(X)$ coordinate-wise by
    \begin{equation*}
        P_i Lx=\sum_{j<i}S_{i,j}x_j \qquad (x=(x_j)_{j\in\N}\in\ell_p(X),\ i\in\N).
    \end{equation*}
    For $x=(x_j)_{j\in\N}\in\ell_p(X)$, we have
    \begin{equation*}
        Lx=\sum_{j=1}^\infty H_jx_j.
    \end{equation*}
    Hence
    \begin{equation*}
        \norm{Lx}\leq \sum_{j=1}^\infty \norm{H_jx_j}\leq \sum_{j=1}^\infty (\varepsilon/2^{j+1})\norm{x_j}\leq \left(\sum_{j=1}^\infty \varepsilon/2^{j+1}\right)\norm{x}= (\varepsilon/2)\norm{x}.
    \end{equation*}
    Thus $L$ is bounded and $\norm{L}<\varepsilon$. Now set $U=S-L$. Since $L$ is precisely the strictly lower triangular part of $S$, the operator $U$ is upper triangular. Moreover,
    \begin{equation*}
        \norm{R_{\mathbf{n},\mathbf{m}}TJ_{\mathbf{n},\mathbf{m}}-U}=\norm{S-U}=\norm{L}<\varepsilon.
    \end{equation*}
\end{proof}

\subsection{Reduction to lower triangular form.}

We next establish a corresponding reduction to lower triangular form. Together with the preceding reduction, this will yield the desired diagonal form. The argument rests on the topological indivisibility, together with the contrast between the inherently $\ell_1$-type nature of scalar addition and the fact that disjoint vectors in $\ell_p$ cannot combine in an $\ell_1$-fashion whenever $1 < p \leq \infty$.

We prove a reduction for row operators to essentially disjoint parts, in the spirit of \Cref{cor: preliminary-upper-triangular} for column operators. We control the norm of the resulting row operator, rather than the sum of the norms of its entries. Recall that, if $A$ is a closed subset of $[0,\alpha]$ or $[0,1]$, then $R_A$ denotes the corresponding restriction operator.

\begin{lemma}[Backward reduction] \label{lmm: lower-triangular}
    Let $K$ be either $[0,1]$ or an ordinal interval $[0,\alpha]$ with characteristic system $(\eta,1)$, and let $1<p\leq\infty$. Let $T\colon \ell_p(C(K))\to C(K)$ be an operator, written as $T=(T_m)_{m\in\N}$, and let $\varepsilon>0$. Then there exist an infinite subset $M\subseteq\N$ and a closed subset $A\subseteq K$ homeomorphic to $K$ such that
    \begin{equation*}
        \norm{R_A T J_M} \leq \varepsilon.
    \end{equation*}
\end{lemma}

\begin{proof}
    Let $K$ be either $[0,\alpha]$ or $[0,1]$. We first consider the case $1<p<\infty$. Choose $N \in\N$ satisfying $N^{\frac{p-1}{p}}>\norm{T}/\varepsilon$ and partition $\N$ into infinite disjoint sets $M_1,\dots,M_N$.
    
    For each $1\leq n \leq N$, define
    \begin{equation*}
        A_n=\{a\in K:\sup_{f\in\ell_p(M_n,C(K)),\,\norm{f}\leq 1}|Tf(a)|\leq\varepsilon\}.
    \end{equation*}
    Then each $A_n$ is a closed subset of $K$. We claim that $\bigcup_{n=1}^N A_n=K$. Indeed, suppose towards a contradiction that there exists $a\in K\setminus\bigcup_{n=1}^N A_n$. By definition, for each $n=1,\dots,N$, we can find $f_n\in\ell_p(M_n,C(K))$ with $\norm{f_n}=1$ and $Tf_n(a)>\varepsilon$. Define $f=\sum_{n=1}^N f_n$. Since the $f_n$ have pairwise disjoint supports, $\norm{f}=N^{1/p}$. On the other hand,
    \begin{equation*}
        N \varepsilon<\sum_{n=1}^N Tf_n(a)=Tf(a)\leq\norm{T}\norm{f}=\norm{T}N^{1/p},
    \end{equation*}
    which contradicts the choice of $N$. Hence $\bigcup_{n=1}^N A_n=K$.
    
    It remains only to use the appropriate topological indivisibility statement. If $K=[0,\alpha]$, then Lemma~\ref{lmm: indivisibility-compact-subsets} gives $1\leq n_0\leq N$ such that $A_{n_0}$ is homeomorphic to $[0,\alpha]$. Taking $M=M_{n_0}$ and $A=A_{n_0}$ gives the desired conclusion. If $K=[0,1]$, then Lemma~\ref{lmm: countable-top-indiv} gives $1\leq n_0\leq N$ and a closed subset $A\subseteq A_{n_0}$ homeomorphic to $[0,1]$. Taking $M=M_{n_0}$ gives the desired conclusion, since the defining estimate for $A_{n_0}$ is inherited by closed subsets.
    
    The case $p=\infty$ is entirely analogous. One chooses $N >\norm{T}/\varepsilon$ instead, and observes that, for $f=\sum_{n=1}^N f_n$ with the $f_n$ supported on pairwise disjoint sets $M_n$, we have $\norm{f}=1$. The same contradiction argument gives $\bigcup_{n=1}^N A_n = K$, and the same topological argument finishes the proof.
\end{proof}

\begin{remark}\label{rem: lower-triangular-form}
    One can, in fact, obtain a stronger version of the preceding lemma. Indeed, if, in the proof, one partitions $\N$ into countably many pairwise disjoint subsets $(M_j)_{j\in\N}$, then the same ideas show that the conclusion remains valid for the $E$-sum of $C(\alpha)$, provided that $E$ is a Banach space with an unconditional basis which does not contain the spaces $\ell_1^n$ uniformly as block subspaces. In the case of $C[0,1]$, using the countable topological indivisibility of $[0,1]$, one obtains the corresponding conclusion under the weaker assumption that no block sequence of the basis of $E$ spans a copy of $\ell_1$.
\end{remark}

The previous lemma is obviously false for $\ell_1$-sums. Indeed, for example, the operator
\begin{equation*}
    T \colon \ell_1(C(\alpha)) \to C(\alpha),
    \qquad
    (f_n)_{n \in \N} \mapsto \sum_{n \in \N} f_n
\end{equation*}
clearly cannot satisfy the conclusion of the lemma above. Thus, a different argument will be needed in order to obtain a lower triangular reduction in the case $p = 1$. 

We obtain the following reduction to lower triangular form. We note that, when $p = \infty$, the lower triangular operator associated with a given operator is well defined, has norm no greater than that of the original operator, and is lower triangular in the sense of \Cref{def: lower-triangular-special}. By contrast, when $1 < p < \infty$, the lower triangular part is, a priori, only a linear map; its boundedness is established implicitly through the proof of the estimate $\norm{R P_M T J_M - L} < \varepsilon$.

\begin{proposition}[Reduction to lower triangular operators]\label{prop: lower-triangular}
    Let $K$ be either $[0,1]$ or an ordinal interval $[0,\alpha]$ with characteristic system $(\eta,1)$, and let $1<p\leq\infty$. Then, for every operator
    \begin{equation*}
        T\colon \ell_p(C(K))\to \ell_p(C(K)), \qquad T=(T_{n,m})_{n,m\in\N},
    \end{equation*}
    and every $\varepsilon>0$, there exist an infinite subset $M=\{m_j:j\in\N\}\subseteq\N$ and a family $(A_{m_j})_{j\in\N}$ of closed subsets of $K$, each homeomorphic to $K$, such that
    \begin{equation*}
        \norm{R P_M T J_M-L}<\varepsilon,
    \end{equation*}
    where
    \begin{equation*}
        R\colon \ell_p(M,C(K))\to \left(\bigoplus_{j\in\N}C(A_{m_j})\right)_{\ell_p}
    \end{equation*}
    is the diagonal operator defined by
    \begin{equation*}
        R=\diag(R_{A_{m_j}}:j\in\N),
    \end{equation*}
    and
    \begin{equation*}
        L\colon \ell_p(M,C(K))\to \left(\bigoplus_{j\in\N}C(A_{m_j})\right)_{\ell_p}
    \end{equation*}
    is the lower triangular operator defined by
    \begin{equation*}
        L_{m_i,m_j}=\begin{cases} R_{A_{m_i}}T_{m_i,m_j}, & \text{if } i\geq j,\\ 0, & \text{if } i<j. \end{cases}
    \end{equation*}
\end{proposition}

\begin{proof}
    Let $T$ and $\varepsilon>0$ be fixed. We apply \Cref{lmm: lower-triangular} recursively to the rows of $T$. Choose $m_1=1$. Applying \Cref{lmm: lower-triangular} to the operator
    \begin{equation*}
        P_{m_1}T\colon \ell_p(C(K))\to C(K),
    \end{equation*}
    we obtain an infinite subset $M_1\subseteq\N$ and a closed subset $A_{m_1}\subseteq K$, homeomorphic to $K$, such that
    \begin{equation*}
        \norm{R_{A_{m_1}}P_{m_1}TJ_{M_1}}<\varepsilon/4.
    \end{equation*}

    Suppose now that we have constructed
    \begin{equation*}
        1=m_1<\dots<m_k, \qquad \N=M_0\supseteq M_1\supseteq\dots\supseteq M_k,
    \end{equation*}
    and closed sets $A_{m_1},\dots,A_{m_k}\subseteq K$, each homeomorphic to $K$, such that $m_j\in M_{j-1}$ and
    \begin{equation*}
        \norm{R_{A_{m_j}}P_{m_j}TJ_{M_j}}<\varepsilon/2^{j+1}
    \end{equation*}
    for each $j\leq k$.

    Choose any $m_{k+1}\in M_k$ with $m_{k+1}>m_k$. Identifying $\ell_p(M_k,C(K))$ canonically with $\ell_p(C(K))$, we may apply \Cref{lmm: lower-triangular} to
    \begin{equation*}
        P_{m_{k+1}}TJ_{M_k}\colon \ell_p(M_k,C(K))\to C(K).
    \end{equation*}
    We thus obtain an infinite subset $M_{k+1}\subseteq M_k$ and a closed subset $A_{m_{k+1}}\subseteq K$, homeomorphic to $K$, such that
    \begin{equation*}
        \norm{R_{A_{m_{k+1}}}P_{m_{k+1}}TJ_{M_{k+1}}}<\varepsilon/2^{k+2}.
    \end{equation*}
    This completes the inductive construction.

    Set $M=\{m_j:j\in\N\}$, and define $R$ and $L$ as in the statement. It is straightforward to verify that $L$ is a linear map; its boundedness will follow from the estimates below. Let $W=RP_MTJ_M-L$. For each $j\in\N$, let
    \begin{equation*}
        Q_{m_j}\colon \left(\bigoplus_{i\in\N}C(A_{m_i})\right)_{\ell_p}\to C(A_{m_j})
    \end{equation*}
    denote the coordinate projection onto the $m_j$-th coordinate. Then, by the definition of $L$, for every $x\in \ell_p(M,C(K))$ we have
    \begin{equation*}
        Q_{m_j}Wx=R_{A_{m_j}}P_{m_j}TJ_M P_{M\setminus\{m_1,\dots,m_j\}}x,
    \end{equation*}
    where $P_{M\setminus\{m_1,\dots,m_j\}}$ denotes the coordinate projection on $\ell_p(M,C(K))$ onto the tail coordinates. Since $M\setminus\{m_1,\dots,m_j\}\subseteq M_j$, it follows that
    \begin{equation*}
        \norm{Q_{m_j}W}\leq\norm{R_{A_{m_j}}P_{m_j}TJ_{M_j}}<\varepsilon/2^{j+1}
    \end{equation*}
    for every $j\in\N$. Finally, let $x\in\ell_p(M,C(K))$ with $\norm{x}\leq 1$. Then
    \begin{align*}
        \norm{Wx}&=\left\|\bigl(Q_{m_j}Wx\bigr)_{j\in\N}\right\|_{\ell_p}\leq\sum_{j=1}^\infty\norm{Q_{m_j}Wx}\leq\sum_{j=1}^\infty\norm{Q_{m_j}W}\norm{x} \\&<\sum_{j=1}^\infty\varepsilon/2^{j+1}=\varepsilon/2.
    \end{align*}
    Hence $\norm{RP_MTJ_M-L}=\norm{W}<\varepsilon$. Since $RP_MTJ_M$ and $W$ are bounded, $L=RP_MTJ_M-W$ is bounded as well.
\end{proof}

\begin{remark}
    In the previous proposition, when we refer to a lower triangular operator in the case $p=\infty$, we mean this in the sense developed in \Cref{sec: notation}. That is, it is not enough for the operator merely to have a lower triangular operator matrix; it must also genuinely act according to that matrix.
\end{remark}

\subsection{Dichotomy for upper triangular operators.}

Finally, the coup de gr\^ace: we show that, if the operator $T$ is upper triangular, then either $T$ or $I-T$ factors the identity on the space. For convenience, we restrict ourselves to the case $1<p<\infty$. We shall deal with the cases $p = 1$ and $p = \infty$ later.

\begin{proposition}[Dichotomy for upper triangular operators]\label{prop: upper-triangular-dichotomy}
    Let $1<p<\infty$, and let $K$ be either $[0,1]$ or an ordinal interval $[0,\alpha]$ with characteristic system $(\eta,1)$. Suppose that $C(K)$ has the $C$-PFP, for some constant $C\geq 1$. Then, for every upper triangular operator $U\colon \ell_p(C(K))\to\ell_p(C(K))$, either $U$ or $I_{\ell_p(C(K))}-U$ factors the identity on $\ell_p(C(K))$ with constant $2C$.
\end{proposition}

\begin{proof}
    Fix an upper triangular operator $U\colon \ell_p(C(K))\to\ell_p(C(K))$ and choose $\delta>0$ such that $C\delta<1$ and
    \begin{equation*}
        \frac{C}{1-C\delta}<2C.
    \end{equation*}
    Applying \Cref{prop: lower-triangular}, we obtain an infinite subset
    \begin{equation*}
        M=\{m_j:j\in\N\}\subseteq\N
    \end{equation*}
    and closed sets $A_{m_j}\subseteq K$, each homeomorphic to $K$, such that
    \begin{equation*}
        \norm{R P_M U J_M-L}<\delta,
    \end{equation*}
    where $L$ is lower triangular. Since $U$ is upper triangular, $L$ is in fact diagonal.

    For each $j\in\N$, let $V_j\colon C(K)\to C(A_{m_j})$ be the isometric isomorphism induced by a homeomorphism between $A_{m_j}$ and $K$, and let $S_j\colon C(A_{m_j})\to C(K)$ be a norm-one extension operator, so that
    \begin{equation*}
        R_{A_{m_j}}S_j=I_{C(A_{m_j})}.
    \end{equation*}
    Define
    \begin{equation*}
        V=\diag(V_j:j\in\N) \qquad \text{and} \qquad S=\diag(S_j:j\in\N).
    \end{equation*}
    Consider the operator
    \begin{equation*}
        W=V^{-1}R P_M U J_M S V\colon \ell_p(C(K))\to\ell_p(C(K)).
    \end{equation*}
    Since $\norm{S}=\norm{V}=\norm{V^{-1}}=1$, it follows that
    \begin{equation*}
        \norm{W-D}<\delta,
    \end{equation*}
    where
    \begin{equation*}
        D=V^{-1}LSV.
    \end{equation*}
    Because $L$ is diagonal, so is $D$, say
    \begin{equation*}
        D=\diag(D_j:j\in\N),
    \end{equation*}
    where
    \begin{equation*}
        D_j=V_j^{-1}R_{A_{m_j}}U_{m_j,m_j}S_jV_j\colon C(K)\to C(K).
    \end{equation*}

    Since $C(K)$ has the $C$-PFP, for each $j\in\N$ either $D_j$ or $I_{C(K)}-D_j$ factors $I_{C(K)}$ with factorisation constant at most $C$. After passing to a further infinite subset of $M$ and relabelling, we may suppose that the same alternative holds for every $j\in\N$.

    First, suppose that, for every $j\in\N$, there exist operators $\Phi_j,\Psi_j\colon C(K)\to C(K)$ such that
    \begin{equation*}
        \Phi_jD_j\Psi_j=I_{C(K)} \qquad \text{and} \qquad \norm{\Phi_j}\norm{\Psi_j}\leq C.
    \end{equation*}
    Rescaling if necessary, we may assume that
    \begin{equation*}
        \norm{\Phi_j}\leq \sqrt C \qquad \text{and} \qquad \norm{\Psi_j}\leq \sqrt C \qquad \text{for all } j\in\N.
    \end{equation*}
    Define
    \begin{equation*}
        \Phi=\diag(\Phi_j:j\in\N) \qquad \text{and} \qquad \Psi=\diag(\Psi_j:j\in\N).
    \end{equation*}
    Then
    \begin{equation*}
        \Phi D\Psi=I_{\ell_p(C(K))} \qquad \text{and} \qquad \norm{\Phi}\norm{\Psi}\leq C,
    \end{equation*}
    hence
    \begin{equation*}
        \Phi W\Psi=I_{\ell_p(C(K))}+E, \qquad \text{where} \qquad E=\Phi(W-D)\Psi.
    \end{equation*}
    Since
    \begin{equation*}
        \norm{E}\leq \norm{\Phi}\norm{W-D}\norm{\Psi}\leq C\delta<1,
    \end{equation*}
    it follows that $I_{\ell_p(C(K))}+E$ is invertible, and
    \begin{equation*}
        I_{\ell_p(C(K))}=(I_{\ell_p(C(K))}+E)^{-1}\Phi W\Psi.
    \end{equation*}
    Thus $W$ factors the identity with factorisation constant at most
    \begin{equation*}
        \frac{C}{1-C\delta}<2C.
    \end{equation*}
    Since
    \begin{equation*}
        W=V^{-1}R P_M U J_M S V,
    \end{equation*}
    it follows that $W$ factors through $U$ with constant one. Hence $U$ factors the identity with the same bound.

    The case where, for every $j\in\N$, the operator $I_{C(K)}-D_j$ factors $I_{C(K)}$ is identical. Indeed, in that case, we obtain diagonal operators $\Phi$ and $\Psi$ with
    \begin{equation*}
        \Phi(I_{\ell_p(C(K))}-D)\Psi=I_{\ell_p(C(K))} \qquad \text{and} \qquad \norm{\Phi}\norm{\Psi}\leq C.
    \end{equation*}
    Since
    \begin{equation*}
        \norm{(I_{\ell_p(C(K))}-W)-(I_{\ell_p(C(K))}-D)}=\norm{W-D}<\delta,
    \end{equation*}
    the same perturbation argument shows that $I_{\ell_p(C(K))}-W$ factors the identity with factorisation constant at most
    \begin{equation*}
        \frac{C}{1-C\delta}<2C.
    \end{equation*}
    Since
    \begin{equation*}
        I_{\ell_p(C(K))}-W=V^{-1}R P_M(I_{\ell_p(C(K))}-U)J_M S V,
    \end{equation*}
    it follows that $I_{\ell_p(C(K))}-W$ factors through $I_{\ell_p(C(K))}-U$ with constant one. Hence $I_{\ell_p(C(K))}-U$ factors the identity with the same bound. This completes the proof.
\end{proof}

%% file: 3-0-1-The_ell_1_case.tex
\section{\texorpdfstring{Lower triangular reduction for $\ell_1$-sums}{Lower triangular reduction for ell1-sums}}\label{sec: ell1}

As observed above, \Cref{lmm: lower-triangular} fails when $p=1$, and therefore the lower triangular reduction used in the previous section cannot be applied in this case. We replace it with a different reduction, which uses the special geometry of the outer $\ell_1$-sum together with the extra room provided by an isomorphism $E(X)\simeq X$, where $E$ is either $c_0$ or $\ell_p$ for some $1<p<\infty$. The argument is similar in spirit to the use of the $c_0$-coordinate structure in \Cref{subsec: C01forward}.

Since this can be done in this more abstract setting without any additional work, and since the method may be useful elsewhere, we present the general form below. We begin with an abstract backward reduction, which is the key ingredient in the reduction to lower triangular form.

\begin{lemma}[Backward reduction]\label{lmm:coordinate-compression-l1}
    Let $X$ be a Banach space, let $N\in\N$, let $\varepsilon>0$, and let $E$ be either $c_0$, $\ell_\infty$, or $\ell_p$ for some $1<p<\infty$. Let $T\colon E(X)\to \ell_1^N(X)$ be an operator. Suppose that $N\varepsilon>\norm{T}$ if $E=c_0$ or $E=\ell_\infty$, and that $ N^{(p-1)/p}\varepsilon>\norm{T}$ if $E=\ell_p$ with $1<p<\infty$. Then there exist $1\leq n\leq N$ and an infinite subset $M\subseteq\N$ such that
    \begin{equation*}
        \norm{P_nTJ_M}\leq \varepsilon.
    \end{equation*}
\end{lemma}

\begin{proof}
    We prove the case $c_0(X)$ first. We proceed by contradiction and assume that the statement is false. Thus, for every $1\leq s\leq N$ and every infinite set $M\subseteq\N$, we have $\norm{P_sTJ_M}>\varepsilon$. Let $M_1,\ldots,M_N$ be pairwise disjoint infinite subsets of $\N$.

    By assumption, for each $s=1,\ldots,N$ we have $\norm{P_sTJ_{M_s}}>\varepsilon$. Hence we may choose $x_s\in B_{c_0(M_s,X)}$ and $x_s^*\in B_{X^*}$ such that
    \begin{equation*}
        |x_s^*P_sTJ_{M_s}x_s|>\varepsilon.
    \end{equation*}
    For each choice of signs $\theta=(\theta_1,\ldots,\theta_N)\in\{-1,1\}^N$, set
    \begin{equation*}
        x_\theta=\sum_{s=1}^N \theta_sJ_{M_s}x_s.
    \end{equation*}
    Since the sets $M_1,\ldots,M_N$ are pairwise disjoint, we have $\norm{x_\theta}_{c_0(X)}\leq 1$. Thus $\norm{Tx_\theta}_{\ell_1^N(X)}\leq\norm{T}$ for every choice of signs, and averaging gives
    \begin{equation*}
        \norm{T}\geq \mathbb E_\theta[\norm{Tx_\theta}_{\ell_1^N(X)}].
    \end{equation*}
    Since the norm on $\ell_1^N(X)$ is given by summing the coordinate norms, we have
    \begin{equation*}
        \mathbb E_\theta[\norm{Tx_\theta}_{\ell_1^N(X)}]=\sum_{s=1}^N \mathbb E_\theta[\norm{P_sTx_\theta}]\geq \sum_{s=1}^N \mathbb E_\theta[|x_s^*P_sTx_\theta|].
    \end{equation*}
    For each fixed $s$, since $|\theta_s|=1$, Jensen's inequality gives
    \begin{equation*}
        \mathbb E_\theta[|x_s^*P_sTx_\theta|]\geq \left|\mathbb E_\theta[\theta_s x_s^*P_sTx_\theta]\right|.
    \end{equation*}
    Moreover,
    \begin{equation*}
        \mathbb E_\theta[\theta_s x_s^*P_sTx_\theta]=\sum_{r=1}^N \mathbb E_\theta[\theta_s\theta_r]x_s^*P_sTJ_{M_r}x_r=x_s^*P_sTJ_{M_s}x_s,
    \end{equation*}
    because $\mathbb E_\theta[\theta_s\theta_r]=0$ for $s\neq r$, while $\mathbb E_\theta[\theta_s^2]=1$. Hence
    \begin{equation*}
        \mathbb E_\theta[|x_s^*P_sTx_\theta|]\geq |x_s^*P_sTJ_{M_s}x_s|>\varepsilon.
    \end{equation*}
    Therefore
    \begin{equation*}
        \norm{T}\geq \mathbb E_\theta[\norm{Tx_\theta}_{\ell_1^N(X)}]>N\varepsilon,
    \end{equation*}
    contradicting the assumption $N\varepsilon>\norm{T}$.

    The proof for $\ell_p(X)$, $1<p<\infty$, is identical, except that now $x_s\in B_{\ell_p(M_s,X)}$ and $\norm{x_\theta}_{\ell_p(X)}\leq N^{1/p}$. Hence
    \begin{equation*}
        \mathbb E_\theta[\norm{Tx_\theta}_{\ell_1^N(X)}]\leq N^{1/p}\norm{T}.
    \end{equation*}
    The same lower estimate gives $\mathbb E_\theta[\norm{Tx_\theta}_{\ell_1^N(X)}]>N\varepsilon$, contradicting $N^{(p-1)/p}\varepsilon>\norm{T}$. Finally, the case $\ell_\infty(X)$ is identical to the case $c_0(X)$, since disjoint supports give $\norm{x_\theta}_{\ell_\infty(X)}\leq 1$.
\end{proof}

Using this version of backward reduction, we can give a reduction to lower triangular operators. Recall that given two sequences $\mathbf{n} = (n_j)_{j\in\N}$ and $\mathbf{m} =(m_j)_{j\in\N}$ in $\N$, with $\mathbf{n}$ increasing, we use the notation $J_{\mathbf{n},\mathbf{m}}$ and $R_{\mathbf{n},\mathbf{m}}$ in the natural way, whereas $\widehat{J}_n: c_0(X) \to \ell_1(c_0(X))$ denotes the canonical inclusion into the outer $n$-th coordinate, as introduced in \Cref{subsec: C01forward}. 

\begin{proposition}[Reduction to lower triangular form on $\ell_1(E(X))$]\label{prop:lower-triangular-form-ell1}
Let $X$ be a Banach space and let $E$ be either $c_0$ or $\ell_p$, where $1<p \leq \infty$. Let
\begin{equation*}
    T\colon \ell_1(E(X))\to \ell_1(E(X))
\end{equation*}
be an operator. Then, for every $\varepsilon>0$, there exist increasing sequences $\mathbf{n}=(n_j)_{j\in\N}$ and $\mathbf{m}=(m_j)_{j\in\N}$ in $\N$, and a lower triangular operator $L\colon \ell_1(X)\to \ell_1(X)$, such that
\begin{equation*}
    \norm{R_{\mathbf{n},\mathbf{m}}TJ_{\mathbf{n},\mathbf{m}}-L} <\varepsilon.
\end{equation*}
\end{proposition}

\begin{proof}
    We prove the case $E=c_0$; the case $E=\ell_p$, $1<p \leq \infty$, follows by the same argument, using the second part of \Cref{lmm:coordinate-compression-l1}. Let $T$ and $\varepsilon>0$ be fixed. For each $k\in\N$, choose $N_k\in\N$ such that
    \begin{equation*}
        N_k\varepsilon/2^{k+1}>\norm{T}.
    \end{equation*}

    We use the following convention. If $M\subseteq M'\subseteq\N$ are infinite, then $J_M\colon c_0(M,X)\to c_0(X)$ denotes the canonical inclusion, while $J_M^{M'}\colon c_0(M,X)\to c_0(M',X)$ denotes the canonical inclusion into $c_0(M',X)$.

    We recursively construct increasing sequences $(n_k)_{k\in\N}$ and $(m_k)_{k\in\N}$, infinite subsets
    \begin{equation*}
        \N=A_0\supseteq A_1\supseteq A_2\supseteq\ldots,
    \end{equation*}
    and, for each $k\in\N$ and each $j\in A_k$, infinite subsets $M_j^k\subseteq\N$, such that, for every $k\in\N$,
    \begin{equation*}
        n_k\in A_{k-1}, \qquad A_k\subseteq A_{k-1}\cap\{j\in\N:j>n_k\}, \qquad m_k>m_{k-1},
    \end{equation*}
    where $m_0=0$, and
    \begin{equation*}
        \norm{R_{n_i,m_i}T\widehat J_jJ_{M_j^k}}\leq \varepsilon/2^{i+1} \qquad (1\leq i\leq k,\ j\in A_k).
    \end{equation*}

    Suppose that the construction has been carried out up to stage $k-1$. Choose distinct elements $a_1<\ldots<a_{N_k}$ of $A_{k-1}$. If $k=1$, set $M_j^0=\N$ for every $j\in A_0$. For each $1\leq s\leq N_k$, choose
    \begin{equation*}
        b_s\in M_{a_s}^{k-1} \qquad\text{with}\qquad b_s>m_{k-1}.
    \end{equation*}
    Put
    \begin{equation*}
        B=\{j\in A_{k-1}:j>a_{N_k}\}.
    \end{equation*}
    For each $j\in B$, define $T_j\colon c_0(M_j^{k-1},X)\to \ell_1^{N_k}(X)$ by
    \begin{equation*}
        T_jx=\bigl(R_{a_s,b_s}T\widehat J_jJ_{M_j^{k-1}}x\bigr)_{s=1}^{N_k}.
    \end{equation*}
    We have $\norm{T_j}\leq\norm{T}$. Hence, by \Cref{lmm:coordinate-compression-l1}, applied after identifying $c_0(M_j^{k-1},X)$ with $c_0(X)$, there exist $1\leq s(j)\leq N_k$ and an infinite subset $L_j\subseteq M_j^{k-1}$ such that
    \begin{equation*}
        \norm{R_{a_{s(j)},b_{s(j)}}T\widehat J_jJ_{M_j^{k-1}}J_{L_j}^{M_j^{k-1}}}\leq \varepsilon/2^{k+1}.
    \end{equation*}
    Since $J_{M_j^{k-1}}J_{L_j}^{M_j^{k-1}}=J_{L_j}$, this is
    \begin{equation*}
        \norm{R_{a_{s(j)},b_{s(j)}}T\widehat J_jJ_{L_j}}\leq \varepsilon/2^{k+1}.
    \end{equation*}

    By the pigeonhole principle, we can find $1\leq s_k\leq N_k$ and an infinite subset $A_k\subseteq B$ such that $s(j)=s_k$ for every $j\in A_k$. Define
    \begin{equation*}
        n_k=a_{s_k}, \qquad m_k=b_{s_k}.
    \end{equation*}
    For $j\in A_k$, define $M_j^k=L_j$. Since $M_j^k\subseteq M_j^{k-1}$, all estimates obtained at earlier stages are preserved. The estimate for the new row follows from the choice of $A_k$. This completes the recursive construction.

    Define $S: \ell_1(X) \to \ell_1(X)$ by
    \begin{equation*}
        S=R_{\mathbf{n},\mathbf{m}}TJ_{\mathbf{n},\mathbf{m}}.
    \end{equation*}
    Write $S=(S_{i,j})_{i,j\in\N}$ with respect to the natural decomposition of $\ell_1(X)$. For each $i,j \in \N$ with $i<j$, then $n_j \in A_i$ and $m_j \in M_{n_j}^i$. Hence
    \begin{equation*}
        \norm{S_{i,j}}=\norm{R_{n_i,m_i}TJ_{n_j,m_j}}\leq \varepsilon/2^{i+1}.
    \end{equation*}
    Define $H\colon \ell_1(X)\to\ell_1(X)$, the upper triangular part of $S$, by
    \begin{equation*}
        P_iHx=\sum_{j>i}S_{i,j}x_j \qquad (x=(x_j)_{j\in\N}\in\ell_1(X)).
    \end{equation*}
    Then, for every $x\in\ell_1(X)$, we have
    \begin{equation*}
        \norm{Hx}_{\ell_1(X)}\leq \sum_{i=1}^\infty\sum_{j>i} \varepsilon/2^{i+1}\norm{x_j}\leq \left(\sum_{i=1}^\infty \varepsilon/2^{i+1}\right)\norm{x}_{\ell_1(X)} = (\varepsilon/2)\norm{x}_{\ell_1(X)}.
    \end{equation*}
    Thus $\norm{H}<\varepsilon$. Therefore, if we let $L=S-H$, then $L$ is lower triangular and
    \begin{equation*}
        \norm{R_{\mathbf{n},\mathbf{m}}TJ_{\mathbf{n},\mathbf{m}}-L}=\norm{H}<\varepsilon.
    \end{equation*}

    In the case $E=\ell_p$, $1<p\leq \infty$, the same proof applies with $c_0(M,X)$ replaced by $\ell_p(M,X)$ throughout, and with $N_k$ chosen so that
    \begin{equation*}
        N_k^{(p-1)/p}\varepsilon/2^{k+1}>\norm{T}
    \end{equation*}
    in the cases $1 < p < \infty$.
\end{proof}

%% file: 3-1_l_infty.tex
\section{Technical results for \texorpdfstring{$\ell_\infty$-sums}{l-infinity}}\label{sec: l-infty}

In the previous sections, we exploited the fact that the operator matrix of $T\colon \allowbreak \ell_p(C(K)) \allowbreak \to \ell_p(C(K))$ encodes all the relevant information about $T$, whenever $1\leq p<\infty$. Even though this is no longer true when $p=\infty$, we will show that this operator matrix is enough for our purposes. Before stating the reduction to diagonal operators, we record the following lemma, which is the counterpart to \Cref{lmm: preliminary-upper-triangular} for the setting $p=\infty$.

\begin{lemma}[Forward reduction]\label{lmm: linfty-forward}
    Let $K$ be either $[0,1]$ or an ordinal interval $[0,\alpha]$ with characteristic system $(\eta,1)$. Let $T\colon \ell_\infty(C(K))\to \ell_\infty(C(K))$ be an operator with operator matrix $T=(T_{n,m})_{n,m\in\N}$. Then, for every $\varepsilon>0$, there exist $n_0\in\N$, an infinite subset $M=\{m_j:j\in\N\}\subseteq\N$, and closed subsets $(A_{m_j})_{j\in\N}$ of $K$, each homeomorphic to $K$, such that
    \begin{equation*}
        \norm{R_{A_{m_j}}T_{m_j,n_0}}\leq\varepsilon \qquad \text{for all } j\in\N.
    \end{equation*}
\end{lemma}

\begin{proof}
    Let $T$ and $\varepsilon>0$ be fixed. Choose $N\in\N$ so that $N>\norm{T}/\varepsilon$. For each $n,m\in\N$ define
    \begin{equation*}
        B(n,m)=\{a\in K:\sup_{f\in B_{C(K)}}|T_{m,n}f(a)|\leq\varepsilon\},
    \end{equation*}
    which is a closed subset of $K$. Let
    \begin{equation*}
        M_n=\{m\in\N:B(n,m)\text{ contains a closed subset homeomorphic to }K\}.
    \end{equation*}
    We claim that there exists $1\leq n_0\leq N$ such that $M_{n_0}$ is infinite. Indeed, suppose by contradiction that $M_n$ is finite for $n=1,\ldots,N$. Then we can find
    \begin{equation*}
        m_0\in\N\setminus\bigcup_{n=1}^N M_n.
    \end{equation*}
    By definition, none of the closed sets $B(n,m_0)$ contains a closed subset homeomorphic to $K$. Hence, by \Cref{lmm: indivisibility-compact-subsets} in the ordinal case, and by \Cref{lmm: countable-top-indiv} in the case $K=[0,1]$, we must have
    \begin{equation*}
        \bigcup_{n=1}^N B(n,m_0)\subsetneq K.
    \end{equation*}
    Choose $a\in K\setminus\bigcup_{n=1}^N B(n,m_0)$. By definition, for each $n=1,\ldots,N$ we can find a norm-one function $f_n\in C(K)$ such that $T_{m_0,n}f_n(a)>\varepsilon$. Let $f=\sum_{n=1}^N J_nf_n$, so that $\norm{f}=1$ since $J_1 f_1,\ldots,J_N f_N$ have disjoint supports. It follows that
    \begin{equation*}
        \varepsilon N<\sum_{n=1}^N T_{m_0,n}f_n(a)= P_{m_0} Tf(a)\leq\norm{Tf}\leq\norm{T}\norm{f}=\norm{T},
    \end{equation*}
    which contradicts our choice of $N$. This proves the claim.

    Choose $1\leq n_0 \leq N$ such that $M_{n_0}$ is infinite, and write $M_{n_0}=\{m_j:j\in\N\}$. For each $j\in\N$, choose a closed subset $A_{m_j}\subseteq B(n_0,m_j)$ homeomorphic to $K$. Since $A_{m_j}\subseteq B(n_0,m_j)$, we have
    \begin{equation*}
        \norm{R_{A_{m_j}}T_{m_j,n_0}}\leq\varepsilon \qquad \text{for all } j\in\N.
    \end{equation*}
    This completes the proof.
\end{proof}

\begin{remark}\label{rem: lemma-works}
    The proof of \Cref{lmm: linfty-forward} remains valid for operators
    \begin{equation*}
        T\colon \ell_\infty(C(K))\to\left(\bigoplus_{m\in\N}C(A_{m})\right)_{\ell_\infty},
    \end{equation*}
    provided that each $A_{m}$ is homeomorphic to $K$. In that case, one obtains an infinite set $M=\{m_j:j\in\N\}$ and closed subsets $B_{m_j}\subseteq A_{m_j}$, each homeomorphic to $K$, such that the corresponding restricted coordinate operators are small.
\end{remark}

Using the previous lemma and remark, together with a standard inductive construction, we obtain the following.

\begin{lemma}[Weak upper triangular reduction]\label{lmm: weak-upper}
    Let $K$ be either $[0,1]$ or an ordinal interval $[0,\alpha]$ with characteristic system $(\eta,1)$. Let $T\colon \ell_\infty(C(K))\to \ell_\infty(C(K))$ be an operator with operator matrix $T=(T_{n,m})_{n,m\in\N}$. Then, for every $\varepsilon>0$, there exist a strictly increasing sequence $(m_j)_{j\in\N}$ in $\N$ and closed subsets $(A_{m_j})_{j\in\N}$ of $K$, each homeomorphic to $K$, such that
    \begin{equation*}
        \norm{R_{A_{m_j}}T_{m_j,m_i}}\leq \varepsilon/2^i \qquad \text{for all } i<j.
    \end{equation*}
\end{lemma}

\begin{proof}
    Fix $\varepsilon>0$. For each $n\in\N$, set $A_n^0=K$.

    We construct recursively a strictly increasing sequence $m_1<m_2<\ldots$, nested infinite sets
    \begin{equation*}
        \N=M_0\supseteq M_1\supseteq M_2\supseteq \ldots,
    \end{equation*}
    and, for each $k\in\N$ and each $n\in M_k$, a closed subset $A_n^k$ of $K$, homeomorphic to $K$, such that
    \begin{equation*}
        m_k\in M_{k-1}, \qquad m_k<m_{k+1},
    \end{equation*}
    \begin{equation*}
        A_n^k\subseteq A_n^{k-1} \qquad \text{for all } k\in\N \text{ and all } n\in M_k,
    \end{equation*}
    and
    \begin{equation*}
        \norm{R_{A_n^k}T_{n,m_k}}\leq \varepsilon/2^k \qquad \text{for all } k\in\N \text{ and all } n\in M_k.
    \end{equation*}

    To begin, apply \Cref{lmm: linfty-forward} to $T$ with $\varepsilon/2$. We obtain $m_1\in\N$, an infinite set $M_1\subseteq\N$, and, for each $n\in M_1$, a closed subset $A_n^1$ of $K$, homeomorphic to $K$, such that
    \begin{equation*}
        \norm{R_{A_n^1}T_{n,m_1}}\leq \varepsilon/2 \qquad \text{for all } n\in M_1.
    \end{equation*}

    Suppose now that, for some $k\in\N$, we have already constructed $m_1<\ldots<m_k$, infinite sets
    \begin{equation*}
        \N=M_0\supseteq M_1\supseteq \ldots \supseteq M_k,
    \end{equation*}
    and closed sets $A_n^i$ for $i=1,\ldots,k$ and $n\in M_i$, satisfying the above conditions.

    Since $M_k$ is infinite, the set
    \begin{equation*}
        M_k'=M_k\cap\{n\in\N:n>m_k\}
    \end{equation*}
    is infinite. Identify $M_k'$ with $\N$ via the increasing enumeration, and apply \Cref{lmm: linfty-forward} to the corresponding compression of $T$ to $M_k'$. By \Cref{rem: lemma-works}, we may do so with the range at coordinate $n\in M_k'$ restricted to $A_n^k$. We obtain $m_{k+1}\in M_k'$, an infinite set $M_{k+1}\subseteq M_k'$, and, for each $n\in M_{k+1}$, a closed subset $A_n^{k+1}$ of $A_n^k$, homeomorphic to $K$, such that
    \begin{equation*}
        \norm{R_{A_n^{k+1}}T_{n,m_{k+1}}}\leq \varepsilon/2^{k+1} \qquad \text{for all } n\in M_{k+1}.
    \end{equation*}
    This completes the recursive construction.

    Finally, set
    \begin{equation*}
        M=\{m_j:j\in\N\}, \qquad A_{m_j}=A_{m_j}^{j-1} \qquad \text{for all } j\in\N.
    \end{equation*}
    This is well defined because $m_j\in M_{j-1}$.

    Let $1\leq i<j$. Since $m_j\in M_{j-1}\subseteq M_i$, the set $A_{m_j}^i$ is defined and, by construction, $A_{m_j}=A_{m_j}^{j-1}\subseteq A_{m_j}^i$. It follows that
    \begin{equation*}
        \norm{R_{A_{m_j}}T_{m_j,m_i}}\leq \norm{R_{A_{m_j}^i}T_{m_j,m_i}}\leq \varepsilon/2^i.
    \end{equation*}
    This proves the result.
\end{proof}

\begin{lemma}[Lower triangular to diagonal reduction]\label{lmm: lower-triangular-reduction}
    Let $K$ be either $[0,1]$ or an ordinal interval $[0,\alpha]$ with characteristic system $(\eta,1)$ for some ordinal $\eta$. Then, for every lower triangular operator
    \begin{equation*}
        L\colon \ell_\infty(C(K))\to \ell_\infty(C(K)), \qquad \text{with operator matrix } L=(L_{n,m})_{n,m\in\N},
    \end{equation*}
    and every $\varepsilon>0$, there exist an infinite subset $M=\{m_j:j\in\N\}\subseteq\N$ and a family $(A_{m_j})_{j\in\N}$ of closed subsets of $K$, each homeomorphic to $K$, such that
    \begin{equation*}
        \norm{R P_M L J_M-D}\leq\varepsilon,
    \end{equation*}
    where
    \begin{equation*}
        R\colon \ell_\infty(M,C(K))\to \left(\bigoplus_{j\in\N}C(A_{m_j})\right)_{\ell_\infty}
    \end{equation*}
    is the diagonal operator defined by
    \begin{equation*}
        R=\diag(R_{A_{m_j}}:j\in\N),
    \end{equation*}
    and
    \begin{equation*}
        D\colon \ell_\infty(M,C(K))\to \left(\bigoplus_{j\in\N}C(A_{m_j})\right)_{\ell_\infty}
    \end{equation*}
    is the diagonal operator defined by
    \begin{equation*}
        D=\diag(R_{A_{m_j}}L_{m_j,m_j}:j\in\N).
    \end{equation*}
\end{lemma}

\begin{proof}
    Suppose that $L$ is lower triangular and let $\varepsilon>0$ be fixed. By \Cref{lmm: weak-upper}, we can find an infinite set $M=\{m_j:j\in\N\}\subseteq\N$ and closed subsets $(A_{m_j})_{j\in\N}$ of $K$, each homeomorphic to $K$, such that
    \begin{equation*}
        \norm{R_{A_{m_i}}L_{m_i,m_j}}\leq\varepsilon/2^j \qquad \text{for all } j<i.
    \end{equation*}
    We claim that defining $R$ and $D$ as in the statement works. To this end, set
    \begin{equation*}
        W=R P_M L J_M-D.
    \end{equation*}
    Then $W$ is lower triangular, since $L$ is lower triangular and $D$ is diagonal. Moreover, for each $i\in\N$, let
    \begin{equation*}
        Q_{m_i}\colon \left(\bigoplus_{j\in\N}C(A_{m_j})\right)_{\ell_\infty}\to C(A_{m_i})
    \end{equation*}
    denote the canonical projection onto the $i$-th coordinate. Since $W$ is lower triangular and $Q_{m_i}WJ_{m_i}=0$ for every $i\in\N$, it follows that, for each fixed $i\in\N$,
    \begin{equation*}
        \norm{Q_{m_i}W}\leq\sum_{j<i}\norm{Q_{m_i}WJ_{m_j}}=\sum_{j<i}\norm{R_{A_{m_i}}L_{m_i,m_j}}\leq\sum_{j<i}\varepsilon/2^j<\varepsilon.
    \end{equation*}
    Therefore,
    \begin{equation*}
        \norm{W}=\sup_{i\in\N}\norm{Q_{m_i}W}\leq\varepsilon,
    \end{equation*}
    which completes the proof.
\end{proof}

\begin{remark}
    This reduction also holds in the case $1<p<\infty$, with an almost identical proof, and one could in fact use it to give a unified treatment of the cases $1<p\leq\infty$. However, the case $p=\infty$ is slightly more delicate: one must take care with the identification of an operator and its matrix, and \Cref{lmm: weak-upper} requires an additional projection onto the sets $A_{m_j}$, whereas its counterpart, \Cref{lmm: lemma-upper-triangular-1}, does not. For the sake of clarity, we have therefore chosen to treat these cases separately.
\end{remark}

We can finally state the dichotomy for lower triangular operators.

\begin{proposition}[Dichotomy for lower triangular operators]\label{prop: lower-triangular-dichotomy}
    Let $K$ be either $[0,1]$ or an ordinal interval $[0,\alpha]$ with characteristic system $(\eta,1)$, and suppose that $C(K)$ has the $C$-PFP for some constant $C\geq 1$. Then, for every lower triangular operator $L\colon \ell_\infty(C(K))\to \ell_\infty(C(K))$, either $L$ or $I_{\ell_\infty(C(K))}-L$ factors the identity on $\ell_\infty(C(K))$ with constant $2C$.
\end{proposition}

\begin{proof}
    The proof is identical to that of \Cref{prop: upper-triangular-dichotomy}, replacing \Cref{prop: lower-triangular} with \Cref{lmm: lower-triangular-reduction}. Indeed, by \Cref{lmm: lower-triangular-reduction}, every lower triangular operator is, after passing to suitable complemented copies of the space, an arbitrarily small perturbation of a diagonal operator. The same diagonal UPFP argument and perturbation argument therefore yield the conclusion. We omit the details.
\end{proof}

%% file: 4_proof_main_theorems.tex
\section{Proof of main theorems}\label{sec: main-theorems}

We now turn to the proof of our main results.

\begin{proof}[Proof of \Cref{th: main-theorem}]
    Let $\alpha$ be any ordinal and note that by \Cref{th: Baker}, we have
    \begin{equation*}
        [0,\alpha] \cong [0,\omega^\eta \cdot n]
    \end{equation*}
    for some ordinal $\eta$ and some $n \in \N$. Consequently,
    \begin{equation*}
        \ell_p(C(\alpha)) \simeq \ell_p(C(\omega^\eta \cdot n)) \simeq \ell_p(C(\omega^\eta)),
    \end{equation*}
    so, without loss of generality, we may assume that $\alpha = \omega^\eta$. In particular, $[0,\alpha]$ has characteristic system $(\eta,1)$. Let such an $\alpha$ be fixed from now on. Further, observe that by \Cref{th: alspach-benyamini}, the space $C(\alpha)$ has the UPFP. In particular, it has the $C$-PFP for some constant $C \geq 1$. 
    
    We start by considering the case $1 < p < \infty$. Let $T \colon \ell_p(C(\alpha)) \to \ell_p(C(\alpha))$ be an operator, represented as a matrix $T = (T_{n,m})_{n,m \in \N}$. Choose $\varepsilon > 0$ small enough so that $2C \varepsilon < 1/2$.
    
    By \Cref{prop: alpha-to-ellp-is-compact}, every operator from $C(\alpha)$ to $\ell_p$ is compact. Hence, by \Cref{prop: prop-upper-triangular}, we can find an infinite subset $M \subseteq \N$ such that
    \begin{equation*}
        P_{M} T J_{M} = U + E,
    \end{equation*}
    where $U$ is upper triangular and $\norm{E} < \varepsilon$.

    By \Cref{prop: upper-triangular-dichotomy}, after identifying $\ell_p(M,C(\alpha))$ canonically with $\ell_p(C(\alpha))$, there exist operators $\Psi$ and $\Phi$ with $\norm{\Psi}\norm{\Phi} \leq 2C$ such that either $\Psi U \Phi = I_{\ell_p(C(\alpha))}$ or $\Psi(I_{\ell_p(C(\alpha))} - U)\Phi = I_{\ell_p(C(\alpha))}$.

    Suppose first that $\Psi U \Phi = I_{\ell_p(C(\alpha))}$. Then
    \begin{equation*}
        \Psi P_{M} T J_{M} \Phi = \Psi(U+E)\Phi = I_{\ell_p(C(\alpha))} + \Psi E \Phi.
    \end{equation*}
    Since
    \begin{equation*}
        \norm{\Psi E \Phi} \leq \norm{\Psi}\norm{E}\norm{\Phi} < 2C\varepsilon < 1/2,
    \end{equation*}
    the operator $I_{\ell_p(C(\alpha))} + \Psi E \Phi$ is invertible, and
    \begin{equation*}
        \norm{(I_{\ell_p(C(\alpha))} + \Psi E \Phi)^{-1}}
        \leq
        \frac{1}{1-\norm{\Psi E \Phi}}
        <
        2.
    \end{equation*}
    Therefore,
    \begin{equation*}
        I_{\ell_p(C(\alpha))}
        =
        (I_{\ell_p(C(\alpha))} + \Psi E \Phi)^{-1}
        \Psi P_{M} T J_{M} \Phi.
    \end{equation*}
    Hence
    \begin{equation*}
        I_{\ell_p(C(\alpha))}
        =
        \bigl((I_{\ell_p(C(\alpha))} + \Psi E \Phi)^{-1}\Psi P_{M}\bigr)\,
        T\,
        \bigl(J_{M}\Phi\bigr),
    \end{equation*}
    so $T$ factors the identity. Moreover, the factorisation constant is bounded by
    \begin{equation*}
        \norm{(I_{\ell_p(C(\alpha))} + \Psi E \Phi)^{-1}\Psi P_{M}}\,
        \norm{J_{M}\Phi}
        \leq
        2 \norm{\Psi}\norm{\Phi}
        \leq
        4C.
    \end{equation*}
    The case where $\Psi(I_{\ell_p(C(\alpha))} - U)\Phi = I_{\ell_p(C(\alpha))}$ is identical. Indeed, in that case, one applies the same perturbation argument to
    \begin{equation*}
        I_{\ell_p(C(\alpha))} - P_{M}TJ_{M}
        =
        P_{M}(I_{\ell_p(C(\alpha))}-T)J_{M},
    \end{equation*}
    and concludes that $I_{\ell_p(C(\alpha))} - T$ factors the identity with the same bound.

    In the case $p = \infty$, one first applies \Cref{prop: lower-triangular} to obtain an infinite subset $M = \{m_j : j \in \N\} \subseteq \N$, closed sets $(A_{m_j})_{j \in \N}$ each homeomorphic to $[0,\alpha]$, and a lower triangular operator
    \begin{equation*}
        L \colon \ell_\infty(M,C(\alpha)) \to \left(\bigoplus_{j \in \N} C(A_{m_j})\right)_{\ell_\infty}
    \end{equation*}
    such that
    \begin{equation*}
        \norm{R P_M T J_M - L} < \varepsilon.
    \end{equation*}
    Since each $A_{m_j}$ is homeomorphic to $[0,\alpha]$, after making the canonical identifications we may regard $L$ as a lower triangular operator on $\ell_\infty(C(\alpha))$. By \Cref{prop: lower-triangular-dichotomy}, either $L$ or $I_{\ell_\infty(C(\alpha))} - L$ factors the identity on $\ell_\infty(C(\alpha))$ with constant $2C$. The same perturbation argument as above then shows that either $T$ or $I_{\ell_\infty(C(\alpha))} - T$ factors the identity, with the same type of estimate on the factorisation constant.

    Finally, the deduction of primariness from the UPFP is standard, via Pe\l czy\'nski's decomposition method and the $\ell_p$-stability of the spaces involved.
\end{proof}

\begin{remark}
    The key ingredients in the proofs are \Cref{lmm: preliminary-upper-triangular,lmm: lower-triangular}. The former remains valid for $E$-sums whenever $E$ has a $1$-unconditional basis and contains no copy of $c_0$, while, by \Cref{rem: lower-triangular-form}, the latter remains valid for $E$-sums of $C(\alpha)$ whenever $E$ has a $1$-unconditional basis and does not contain the spaces $\ell_1^n$ uniformly as block subspaces. It follows that the UPFP also holds for $E$-sums of $C(\alpha)$ spaces whenever $E$ has a (sub)symmetric basis and satisfies these additional assumptions.
\end{remark}

\begin{proof}[Proof of \Cref{th: main-theorem2}]
    Suppose first that $1<p<\infty$. Since $c_0(C[0,1])$ is isomorphic to $C[0,1]$, it is enough to prove that $\ell_p(c_0(C[0,1]))$ has the UPFP. By \Cref{prop: C01-UPFP}, the space $C[0,1]$ has the UPFP. In particular, it has the $C$-PFP for some constant $C\geq 1$.

    Let $T\colon \ell_p(c_0(C[0,1]))\to \ell_p(c_0(C[0,1]))$ be an operator, and choose $\varepsilon>0$ such that $2C\varepsilon<1/2$. By \Cref{lmm:lp-c0-upper-reduction}, there exist increasing sequences $\mathbf n=(n_j)_{j\in\N}$ and $\mathbf m=(m_j)_{j\in\N}$ in $\N$, and an upper triangular operator $U\colon \ell_p(C[0,1])\to\ell_p(C[0,1])$, such that
    \begin{equation*}
        \norm{R_{\mathbf n,\mathbf m}TJ_{\mathbf n,\mathbf m}-U}<\varepsilon.
    \end{equation*}
    Set
    \begin{equation*}
        E=R_{\mathbf n,\mathbf m}TJ_{\mathbf n,\mathbf m}-U.
    \end{equation*}
    By \Cref{prop: upper-triangular-dichotomy}, either $U$ or $I_{\ell_p(C[0,1])}-U$ factors the identity on $\ell_p(C[0,1])$ with constant $2C$. Suppose first that $U$ factors the identity. Then there exist operators $\Psi,\Phi\colon \ell_p(C[0,1])\to\ell_p(C[0,1])$ such that
    \begin{equation*}
        \Psi U\Phi=I_{\ell_p(C[0,1])} \qquad \text{and} \qquad \norm{\Psi}\norm{\Phi}\leq 2C.
    \end{equation*}
    Hence
    \begin{equation*}
        \Psi R_{\mathbf n,\mathbf m}TJ_{\mathbf n,\mathbf m}\Phi=I_{\ell_p(C[0,1])}+\Psi E\Phi.
    \end{equation*}
    Since
    \begin{equation*}
        \norm{\Psi E\Phi}\leq \norm{\Psi}\norm{E}\norm{\Phi}<2C\varepsilon<1/2,
    \end{equation*}
    the operator $I_{\ell_p(C[0,1])}+\Psi E\Phi$ is invertible, with inverse of norm at most $2$. Therefore
    \begin{equation*}
        I_{\ell_p(C[0,1])}=(I_{\ell_p(C[0,1])}+\Psi E\Phi)^{-1}\Psi R_{\mathbf n,\mathbf m}T J_{\mathbf n,\mathbf m}\Phi.
    \end{equation*}
    Hence $T$ factors the identity on $\ell_p(C[0,1])$ with factorisation constant at most $4C$.

    The case where $I_{\ell_p(C[0,1])}-U$ factors the identity is identical. Indeed, since
    \begin{equation*}
        I_{\ell_p(C[0,1])}-R_{\mathbf n,\mathbf m}TJ_{\mathbf n,\mathbf m}=R_{\mathbf n,\mathbf m}(I_{\ell_p(c_0(C[0,1]))}-T)J_{\mathbf n,\mathbf m}
    \end{equation*}
    and
    \begin{equation*}
        \norm{(I_{\ell_p(C[0,1])}-R_{\mathbf n,\mathbf m}TJ_{\mathbf n,\mathbf m})-(I_{\ell_p(C[0,1])}-U)}=\norm{E}<\varepsilon,
    \end{equation*}
    the same perturbation argument shows that $I_{\ell_p(c_0(C[0,1]))}-T$ factors the identity on $\ell_p(C[0,1])$ with factorisation constant at most $4C$. Since $c_0(C[0,1])$ is isomorphic to $C[0,1]$, this shows that either $T$ or $I_{\ell_p(c_0(C[0,1]))}-T$ factors the identity on $\ell_p(c_0(C[0,1]))$ with uniform control over the constant, which proves the result.

    The case $p=\infty$ is identical to the proof of \Cref{th: main-theorem}. Briefly, given an operator $T\colon \ell_\infty(C[0,1])\to\ell_\infty(C[0,1])$, one applies \Cref{prop: lower-triangular} to pass, after restricting to suitable closed copies of $[0,1]$, to a lower triangular operator. Then \Cref{prop: lower-triangular-dichotomy} shows that either this lower triangular operator or its complement factors the identity with uniform constant. The same perturbation argument as above then implies that either $T$ or $I_{\ell_\infty(C[0,1])}-T$ factors the identity, again with a uniform constant. Hence $\ell_\infty(C[0,1])$ has the UPFP.

    Primariness follows again from the UPFP and Pe\l czy\'nski's decomposition method, by using the $\ell_p$-stability of the spaces involved.
\end{proof}

\begin{remark}
    The key ingredients in the proofs for the $C[0,1]$ case are \Cref{lmm:c0-lp-coordinate-smallness,lmm: lower-triangular}. By \Cref{rem: upper-triangular-form}, the former remains valid for $E$-sums, that is, for operators $T\colon c_0(X)\to E(Y)$, whenever $E$ contains no copy of $c_0$. On the other hand, by \Cref{rem: lower-triangular-form}, the latter remains valid for $E$-sums of $C[0,1]$ whenever $E$ contains no copy of $\ell_1$. By a classical theorem of James \cite{James1950}, a Banach space with an unconditional basis is reflexive if and only if it contains neither $c_0$ nor $\ell_1$. In particular, the UPFP also holds for $E$-sums of $C[0,1]$ whenever $E$ is reflexive and has a subsymmetric basis.
\end{remark}

We now give the proof for the case $p=1$.

\begin{proof}[Proof of \Cref{th: main-theorem3}]
    Suppose first that $X\simeq c_0(X)$. Since $\ell_1(c_0(X))$ is isomorphic to $\ell_1(X)$, it is enough to prove that $\ell_1(c_0(X))$ has the UPFP. Let $C\geq 1$ be such that $X$ has the $C$-PFP, and choose $\varepsilon>0$ such that $C\varepsilon<1/2$.

    We first note that every operator $S\colon X\to \ell_1$ is compact. Indeed, if some such $S$ were not compact, then $S$ would not be weakly compact, since $\ell_1$ has the Schur property. Thus, by a result of Pe\l czy\'nski \cite{pelczynski1965strictly2}, $S$ would fix a copy of $\ell_1$, which yields a complemented copy of $\ell_1$ in $X$, contradicting the assumption.

    Let $T\colon \ell_1(c_0(X))\to \ell_1(c_0(X))$ be an operator. By \Cref{prop:lower-triangular-form-ell1}, there exist increasing sequences $\mathbf n=(n_j)_{j\in\N}$ and $\mathbf m=(m_j)_{j\in\N}$ in $\N$, and a lower triangular operator $L\colon \ell_1(X)\to\ell_1(X)$, such that
    \begin{equation*}
        \norm{R_{\mathbf n,\mathbf m}TJ_{\mathbf n,\mathbf m}-L}<\varepsilon/2.
    \end{equation*}
    Since every operator $X\to \ell_1$ is compact, \Cref{prop: prop-upper-triangular}, applied to $L\colon \ell_1(X)\to\ell_1(X)$, gives an infinite subset $M\subseteq\N$ and an upper triangular operator $D\colon \ell_1(M,X)\to\ell_1(M,X)$ such that
    \begin{equation*}
        \norm{P_MLJ_M-D}<\varepsilon/2.
    \end{equation*}
    Since $L$ is lower triangular, $D$ is diagonal. Moreover,
    \begin{equation*}
        \norm{P_MR_{\mathbf n,\mathbf m}TJ_{\mathbf n,\mathbf m}J_M-D}\leq \norm{P_M(R_{\mathbf n,\mathbf m}TJ_{\mathbf n,\mathbf m}-L)J_M}+\norm{P_MLJ_M-D}<\varepsilon.
    \end{equation*}

    Write $D=\diag(D_j:j\in M)$. Since $X$ has the $C$-PFP, for each $j\in M$ either $D_j$ or $I_X-D_j$ factors $I_X$ with constant at most $C$. Passing to a further infinite subset of $M$, and relabelling, we may suppose that the same alternative holds for every $j\in M$.

    Suppose first that each $D_j$ factors $I_X$. Then there are operators $\Phi_j,\Psi_j\colon X\to X$ such that
    \begin{equation*}
        \Phi_jD_j\Psi_j=I_X \qquad \text{and} \qquad \norm{\Phi_j}\norm{\Psi_j}\leq C.
    \end{equation*}
    Rescaling if necessary, we may assume that $\norm{\Phi_j}\leq\sqrt C$ and $\norm{\Psi_j}\leq\sqrt C$ for every $j\in M$. Thus the diagonal operators $\Phi=\diag(\Phi_j:j\in M)$ and $\Psi=\diag(\Psi_j:j\in M)$ are bounded on $\ell_1(M,X)$ and satisfy
    \begin{equation*}
        \Phi D\Psi=I_{\ell_1(M,X)} \qquad \text{and} \qquad \norm{\Phi}\norm{\Psi}\leq C.
    \end{equation*}
    Set
    \begin{equation*}
        E=P_MR_{\mathbf n,\mathbf m}TJ_{\mathbf n,\mathbf m}J_M-D.
    \end{equation*}
    Then $\norm{E}<\varepsilon$, and hence $\norm{\Phi E\Psi}<1/2$. Therefore
    \begin{equation*}
        I_{\ell_1(M,X)}=(I_{\ell_1(M,X)}+\Phi E\Psi)^{-1}\Phi P_MR_{\mathbf n,\mathbf m}TJ_{\mathbf n,\mathbf m}J_M\Psi.
    \end{equation*}
    Thus $T$ factors the identity on $\ell_1(M,X)$, with constant bounded independently of $T$.

    The case where each $I_X-D_j$ factors $I_X$ is identical. Indeed, in that case we obtain diagonal operators $\Phi$ and $\Psi$ such that
    \begin{equation*}
        \Phi(I_{\ell_1(M,X)}-D)\Psi=I_{\ell_1(M,X)} \qquad \text{and} \qquad \norm{\Phi}\norm{\Psi}\leq C.
    \end{equation*}
    Since
    \begin{equation*}
        I_{\ell_1(M,X)}-P_MR_{\mathbf n,\mathbf m}TJ_{\mathbf n,\mathbf m}J_M=P_MR_{\mathbf n,\mathbf m}(I_{\ell_1(c_0(X))}-T)J_{\mathbf n,\mathbf m}J_M,
    \end{equation*}
    the same perturbation argument shows that $I_{\ell_1(c_0(X))}-T$ factors the identity on $\ell_1(M,X)$, again with a constant bounded independently of $T$.

    Identifying $\ell_1(M,X)$ with $\ell_1(X)$, and using the fixed isomorphism $\ell_1(c_0(X))\simeq \ell_1(X)$, we conclude that either $T$ or $I_{\ell_1(c_0(X))}-T$ factors the identity on $\ell_1(c_0(X))$ with a uniform constant. Hence $\ell_1(c_0(X))$ has the UPFP, and therefore so does $\ell_1(X)$.

    If $X\simeq \ell_p(X)$ for some $1<p<\infty$, the proof is the same, replacing $c_0(X)$ by $\ell_p(X)$ throughout and using the $\ell_p$-version of \Cref{prop:lower-triangular-form-ell1}. Thus $\ell_1(\ell_p(X))$ has the UPFP, and since $\ell_1(\ell_p(X))\simeq\ell_1(X)$, the space $\ell_1(X)$ has the UPFP.

    Finally, primariness follows from the UPFP and Pe\l czy\'nski's decomposition method.
\end{proof}

\begin{remark}\label{rem: extra-terms}
    The reader may notice that similar results remain valid under more general assumptions. For example, one can also obtain the UPFP, and thus primariness, for the spaces
    \begin{equation*}
        \Bigl( \bigoplus_{n \in \N} C(\omega^{\omega^{\alpha_n}}) \Bigr)_p,
        \qquad 1 \leq p \leq \infty,
    \end{equation*}
    whenever $(\alpha_n)_{n=1}^\infty$ is an increasing sequence of countable ordinals. Further variations of this kind are also possible and follow essentially from the same ideas.
\end{remark}

\begin{remark}
    Naturally, the previous theorems can be extended to finite iterated sums of the form $\ell_{p_1}(\ell_{p_2}(\ldots(\ell_{p_n}(C(K)))\ldots))$. This can be done by combining the techniques of \cite[Section~3]{AcuavivaKania2026} with the ideas developed above. Since we do not currently see any applications of this extension, we do not include the details. Similar considerations will apply to iterated sums of the form in \Cref{rem: extra-terms}.
\end{remark}